\documentclass[10pt]{amsart}
\usepackage{url}
\usepackage{hyperref,amssymb}
\hypersetup{colorlinks=true,linkcolor=blue,
urlcolor=cyan,citecolor=magenta} 
\usepackage[hyperpageref]{backref}
\usepackage{frcursive,calrsfs,mathrsfs}
\usepackage{xcolor}
\newtheorem{theorem}{Theorem}[section]
\newtheorem{lemma}[theorem]{Lemma}
\newtheorem{corollary}[theorem]{Corollary}

\newtheorem{proposition}[theorem]{Proposition}
\newtheorem{remark}[theorem]{Remark}

\newtheorem{definition}[theorem]{Definition}

\numberwithin{equation}{section}
\newcommand{\Q}{{\mathbb{Q}}}
\newcommand{\Z}{{\mathbb{Z}}}
\newcommand{\F}{{\mathbb{F}}}
\newcommand{\ds}{\displaystyle}
\newcommand{\ov}{\overline}

\newcommand{\ft}{\footnotesize}
\newcommand{\ns}{\normalsize}
\newcommand{\order}{\raise0.8pt \hbox{${\scriptstyle \#}$}}

\newcommand{\plus}{\ds\mathop{\raise 0.5pt \hbox{$\bigoplus$}}\limits}
\newcommand{\prd}{\ds\mathop{\raise 1.0pt \hbox{$\prod$}}\limits}
\newcommand{\sm}{\ds\mathop{\raise 1.0pt \hbox{$\sum$}}\limits}
\newcommand{\ffrac}[2]{\hbox{\ft $\displaystyle\frac{#1}{#2}$}}

\newcommand{\Norm}{{\bf N}}
\newcommand{\Trace}{{\bf T}}
\newcommand{\BB}{{\bf B}}

\newcommand{\Bb}{{\bf b}}
\newcommand{\BD}{{\bf \Delta}}

\def\plus{\displaystyle\mathop{\raise 2.0pt \hbox{$\bigoplus $}}\limits}
\def\prd{\displaystyle\mathop{\raise 2.0pt \hbox{$\prod$}}\limits}
\def\sm{\displaystyle\mathop{\raise 2.0pt \hbox{$\sum$}}\limits}

\newcommand{\pow}{{\rm pow}}

\newcommand{\Sgn}{{\bf S}}
\newcommand{\ZK}{{\bf Z}_K}
\newcommand{\fop}{{F.O.P.\,}}

\author[Georges Gras]{Georges Gras}

\address{Villa la Gardette -- 4, chemin de Ch\^ateau Gagni\`ere --
38520 Le Bourg d'Oisans, France -- Url: {\rm \url{http://orcid.org/0000-0002-1318-4414}}}
\email{g.mn.gras@wanadoo.fr}



\begin{document}

\title[Unlimited lists of quadratic integers of given norm]
{Unlimited lists of quadratic integers of given norm \\ 
Application to some arithmetic properties }

\date{May 24, 2023}

\begin{abstract}  We use the polynomials $m_s(t) = t^2 - 4 s$, $s \in \{-1, 1\}$, in an 
elementary process giving unlimited lists of {\it fundamental units of norm $s$},
of real quadratic fields, with ascending order of the discriminants. 
As $t$ grows from $1$ to an upper bound $\BB$, for each {\it first occurrence} 
of a square-free integer $M \geq 2$, in the factorization $m_s(t) =: M r^2$, the
unit $\frac{1}{2} \big(t + r \sqrt{M}\big)$ is the fundamental unit of norm $s$
of $\Q(\sqrt M)$, even if $r >1$ (Theorem \ref{fondunit}).
Using $m_{s\nu}(t) = t^2 - 4 s \nu$, $\nu \geq 2$, the algorithm gives unlimited 
lists of {\it fundamental integers of norm $s\nu$} (Theorem~\ref{unicity}). 
We deduce, for any prime $p>2$, unlimited lists of {\it non $p$-rational} 
quadratic fields (Theorems \ref{prational}, \ref{heuristic}, \ref{infty}) and lists 
of degree $p-1$ imaginary fields with non-trivial $p$-class group (Theorems 
\ref{cubic}, \ref{quintic}). All PARI programs are given. 
\end{abstract}

\maketitle

\tableofcontents

\newpage
\section{Introduction and main results} 

\subsection{Definition of the ``\texorpdfstring{\fop}{Lg}'' algorithm}\label{process}
For the convenience of the reader, we give, at once, an outline of this process which
has an interest especially under the use of PARI programs~\cite{P}. 

\begin{definition}
We call ``{\it First Occurrence Process}'' (\fop) the following algorithm, defined 
on a large interval $[1, \BB]$ of integers.
As $t$ grows from $t=1$ up to $t=\BB$, we compute some arithmetic invariant $F(t)$;
for instance, a pair of invariants described as a PARI list, as the following
illustration with square-free integers $M(t)$ and units $\eta(t)$ of $\Q(\sqrt {M(t)})$:
\[{\sf F(t) \mapsto L(t)= List \big(\big[M(t) ,\  \eta(t) \big]},\]
provided with a natural order on the pairs $L(t)$, then put it in a PARI list ${\sf LM}$:
\begin{equation*}
\begin{aligned}
{\sf Listput(LM,vector(2,c,L[c]))} \mapsto & {\sf List([L(1),L(2),\ldots,L(t), \ldots, L(\BB)])} \\
= & {\sf List([\, [M(1),\eta(1)], \ldots, [M(t),\eta(t)] , \ldots, [M(\BB),\eta(\BB)]\, ])} ;
\end{aligned}
\end{equation*}
after that, we apply the PARI instruction ${\sf VM=vecsort(LM,1 ,8)}$ which builds the list:
\[{\sf VM=List([L_1,L_2,\ldots, L_j, \ldots, L_N]),\ \, {\sf N \leq \BB},}\] 
such that ${\sf L_j = L(t_j) = [\,[M(t_j),\eta(t_j)]\,]}$ is {\it the first occurrence} (regarding the 
selected order, for instance that on the ${\sf M}$'s) of the invariant found by the 
algorithm and which removes the subsequent duplicate entries. 
\end{definition}

Removing the duplicate entries is the key of the principle since in general they are 
unbounded in number as $\BB \to \infty$ and do not give the suitable information

\smallskip
Since the length ${\sf N}$ of the list ${\sf VM}$ is unknown by nature, 
one must write ${\sf LM}$ as a vector and put instead:
${\sf VM=vecsort(vector(\BB,c,LM[c]),1,8)}$; thus, ${\sf N=\#VM}$ makes sense and one 
can (for possible testing) select elements and components as ${\sf X=VM[k][2]}$, etc. 
If ${\sf N}$ is not needed, then ${\sf VM=vecsort(LM,1 ,8)}$ works well.

\smallskip
For instance, the list ${\sf LM}$ 
of objects $F(t) = (M(t), \varepsilon(t))$, $1 \leq t \leq \BB=10$:
\[{\sf LM=List([\, [5,\varepsilon_5^{}],[2,\varepsilon_2^{}],[5,\varepsilon'_5],
[7,\varepsilon_7^{}],[5,\varepsilon''_5],[3,\varepsilon_3^{}],[2,\varepsilon'_2],
[5,\varepsilon'''_5],[6,\varepsilon_6^{}],[7,\varepsilon'_7]\, ])}, \] 
with the natural order on the first components $M$, leads to the list:
\[{\sf VM=List([[2,\varepsilon_2^{}],[3,\varepsilon_3^{}],[5,\varepsilon_5^{}],
[6,\varepsilon_6^{}],[7,\varepsilon_7^{}]])}. \]

\subsection{Quadratic integers}
Let $K =: \Q(\sqrt M)$, $M \in \Z_{\geq 2}$ square-free, be a real quadratic field 
and let $\ZK$ be its ring of integers.

\smallskip
Recall that $M \geq 2$, square-free, is called the ``Kummer radical'' of 
$K$, contrary to any ``radical'' $m=M r^2$ giving the same field $K$.

\smallskip
There are two ways of writing for an element $\alpha \in \ZK$. The first one is to
use the integral basis $\{1, \sqrt M\}$ (resp. $\big\{1, \frac{1+\sqrt M}{2} \big\}$) 
when $M \not \equiv 1 \pmod 4$ (resp. $M \equiv 1 \pmod 4$). The second one
is to write $\alpha = \frac{1}{2}(u+v \sqrt{M})$, in which case $u, v \in \Z$ 
are necessarily of same parity; but $u, v$ may be odd only when
$M \equiv 1 \pmod 4$. 
We denote by $\Trace_{K/\Q}$ and $\Norm_{K/\Q}$, or simply $\Trace$
and $\Norm$, the trace and norm maps in $K/\Q$, so that 
$\Trace(\alpha) = u$ and $\Norm(\alpha) = \frac{1}{4}(u^2 - M v^2)$
in the second writing for $\alpha$. 

\smallskip
Then the norm equation in $u, v \in \Z$ (not necessarily with co-prime numbers $u$, $v$):
\[u^2 - M v^2 = 4s \nu,\ \  s\in \{-1,1\}, \ \nu \in \Z_{\geq 1}, \]
for $M$ square-free, has the property that $u, v$ are necessarily of same parity
and may be odd only when $M \equiv 1 \pmod 4$; then:
\[z := \ffrac{1}{2} \big(u+v \sqrt{M}\big) \in \ZK,\ \ \Trace(z) = u\ \ \&\ \ \Norm (z) = s\nu. \]

Finally, we will write quadratic integers $\alpha$, with positive coefficients on the 
basis $\{1, \sqrt M\}$; this defines a unique representative modulo 
the sign and the conjugation. Put:
\[\ZK^+ := \Big\{ \alpha = \ffrac{1}{2} \big(u+v \sqrt{M}\big),\ \ u, v \in \Z_{\geq 1},
\  u \equiv v \!\!\pmod 2\Big\} .\]

Note that these $\alpha$'s are not in $\Z$, nor in $\Z\! \cdot\! \sqrt M$; indeed, we have the
trivial solutions $\Norm(q) = q^2$ ($\alpha = q \in \Z_{\geq 1}$, $s=1$, $\nu = q^2$,
$M v^2=0$, $u=q$) or $\Norm(v \sqrt M)=-M v^2$, $v \in \Z_{\geq 1}$, $s=-1$, 
$\nu = -M v^2$, $u=0$), which are not given by the \fop algorithm for simplicity. 
These viewpoints will be more convenient for our 
purpose and these conventions will be implicit in all the sequel.

\smallskip
Since norm equations may have several solutions, we will use the following definition:

\begin{definition}\label{defnu}
Let $M \in \Z_{\geq 2}$ be a square-free integer and let $s \in \{-1,1\}$, $\nu \in \Z_{\geq 1}$.
We call fundamental solution (if there are any) of the norm equation $u^2-M v^2 = 4 s\nu$, 
with $u ,v  \in \Z_{\geq 1}$, the corresponding integer $\alpha := \frac{1}{2} (u + v \sqrt M) 
\in \ZK^+$ of minimal trace $u$.
\end{definition}

\subsection{Quadratic polynomial units}
It is classical that the continued fraction expansion of $\sqrt m$, for a positive 
square-free integer $m$, gives, under some limitations, the fundamental solution, 
in integers $u, v \in \Z_{\geq 1}$, of the norm equation $u^2-m v^2=4s$, whence the 
fundamental unit $\varepsilon_m^{}:=\frac{1}{2}(u+v \sqrt{m})$ of $\Q(\sqrt m)$.
A similar context of ``polynomial continued fraction expansion'' does 
exist and gives polynomial solutions $(u(t), v(t))$, of $u(t)^2-m(t) v(t)^2 = 4s$, 
for suitable $m(t) \in \Z[t]$ (see, e.g., \cite{McL,McLZ,Nat,Ram,SaAb}). 
This gives the quadratic polynomial units $E(t) := \frac{1}{2} \big (u(t) + v(t)\sqrt{m(t)} \big)$.

\smallskip
We will base our study on the following polynomials $m(t)$ that have interesting 
universal properties (a first use of this is due to Yokoi \cite[Theorem 1]{Yo}). 

\begin{definition}\label{msdelta}
Consider the square-free polynomials $m_{s\nu}(t) = t^2 - 4 s \nu \in \Z[t]$, 
where $s \in \{-1, 1\}$, $\nu \in \Z_{\geq 1}$. The 
continued fraction expansion of $\sqrt{t^2 - 4 s \nu}$ leads to the integers $A_{s\nu}(t) := 
\frac{1}{2} \big(t+\sqrt{t^2 - 4 s \nu}\big)$, of norm $s \nu$ and trace $t$, in 
a quadratic extension of $\Q(t)$. When $\nu = 1$, one obtains the units 
$E_s(t) := \frac{1}{2} \big(t+\sqrt{t^2 - 4 s}\big)$, of norm $s$ and trace $t$. 
\end{definition}

The continued fraction expansion, with polynomials, gives the fundamental solution 
of the norm equation (cf. details in \cite{McL}), but must not be confused with that 
using evaluations of the polynomials; for instance, for $t_0=7$, $m_1(t_0) = 7^2-4 = 45$ 
is not square-free and $E_1(7) = \frac{1}{2} (7+\sqrt{45}) = \frac{1}{2} (7+3 \sqrt 5)$ is 
indeed the fundamental solution of $u^2-45v^2 = 4$, but not the fundamental unit 
$\varepsilon_5^{}$ of $\Q(\sqrt {45}) = \Q(\sqrt {5})$, since one gets $E_1(7) = 
\varepsilon_5^6$.

\subsection{Main algorithmic results}
We will prove that the families of polynomials $m_{s\nu}(t) = t^2 - 4 s\nu$,
$s\in \{-1,1\}$, $\nu \in \Z_{\geq 1}$, are universal to find all square-free 
integers $M$ for which there exists a privileged solution $\alpha \in \ZK^+$ 
to $\Norm(\alpha) = s \nu$;
moreover, the solution obtained is the fundamental one, in the meaning of 
Definition \ref{defnu} saying that $\alpha$ is of minimal trace $t \geq 1$. 
This is obtained by means of an extremely simple algorithmic process 
(described \S\,\ref{process}) and allows to get unbounded lists of quadratic fields,
given by means of their Kummer radical, and having specific properties.

\smallskip
The typical results, admitting several variations, are given by the following excerpt
of statements using quadratic polynomial expressions $m(t)$ deduced from some 
$m_{s\nu}(t)$:

\begin{theorem}
Let $\BB$ be an arbitrary large upper bound.
As the integer $t$ grows from $1$ up to $\BB$, {\it for each first occurrence} of a  
square-free integer $M \geq 2$, in the factorizations $m(t) =: M r^2$, 
we have the following properties for $K := \Q(\sqrt M)$:

\medskip
{\bf a)} Consider the polynomials $m(t) = t^2 - 4s \nu$, $s \in \{-1,1\}$, $\nu \in \Z_{\geq 2}$:

\medskip
\quad (i) {$m(t) = t^2 - 4s$}. 

\smallskip\noindent
The unit $\frac{1}{2}(t + r \sqrt M)$ is the 
fundamental unit of norm $s$ of $K$ (Theorem \ref{fondunit}).

\smallskip
\quad (ii) $m(t) = t^2 - 4 s \nu$.

\smallskip\noindent
The integer $A_{s\nu}(t) = \frac{1}{2}(t + r \sqrt M)$ is the fundamental integer 
in $\ZK^+$ of norm $s\nu$ in the meaning of Definition \ref{defnu} (Theorem \ref{unicity}).

\medskip
{\bf b)} Let $p$ be an odd prime number and consider the following
polynomials.

\medskip
\quad (i) {$m(t) \in \big\{$$p^4 t^2 \pm 1$, $p^4 t^2 \pm 2$, $p^4 t^2 \pm 4$, 
$4p^4 t^2 \pm 2$, $9 p^4 t^2 \pm 6$, $9 p^4 t^2 \pm 12,\,\ldots$$\big\}$}.

\smallskip\noindent
The field $K$ is non $p$-rational apart from few explicit 
cases (Theorem \ref{prational}).

\smallskip
\quad (ii) {$m(t) = 3^4 t^2 - 4s$}. 

\smallskip\noindent
The field $F_{3,M} := \Q(\sqrt {-3 M})$ has a class number divisible by $3$, except 
possibly when the unit $\frac{1}{2} \big( 9 t + r \sqrt M \big)$ is a third power 
of a unit (Theorem \ref{cubic}).
Up to $\BB=10^5$, all the $3$-class groups are non-trivial, apart from few 
explicit cases.

\smallskip
\quad (iii) {$m(t) = p^4 t^2 - 4s$, $p \geq 5$}. 

\smallskip\noindent
The imaginary cyclic extension 
$F_{p,M} := \Q\big ((\zeta_p-\zeta_p^{-1}) \sqrt M \big)$, of degree $p-1$, has a class number 
divisible by $p$, except possibly when the unit $\frac{1}{2} \big(p^2 t + r \sqrt {M})\big)$ 
is a $p$-th power of unit (Theorem \ref{quintic}). 

\smallskip\noindent
For $p=5$, the quartic cyclic field $F_{5,M}$ is 
defined by the polynomial $P=x^4+5 M x^2+5 M^2$ and up to $\BB=500$, all the 
$5$-class groups are non-trivial, except for $M=29$.
\end{theorem}

\smallskip
Moreover, this principle gives lists of solutions by means of Kummer radicals
(or discriminants) of a regularly increasing order of magnitude, these lists being
unbounded as $\BB \to \infty$. See, for instance Proposition \ref{existence} for lists 
of Kummer radicals $M$, then Section \ref{listes} for lists of arithmetic invariants 
(class groups, $p$-ramified torsion groups, logarithmic class groups of $K$),
and Theorems \ref{heuristic}, \ref{infty}, giving unlimited lists of units, local (but non 
global) $p$th powers, whence lists of non-$p$-rational quadratic fields. 

\smallskip
All the lists have, at least, $O(\BB)$ distinct elements, but most often 
$\BB-o(\BB)$, and even $\BB$ distinct elements in some situations.

\medskip
So, we intend to analyze these results in a computational point of view by means of 
a new strategy to obtain arbitrary large list of fundamental units, or of other quadratic 
integers, {\it even when radicals $m_{s\nu}(t) =: M(t) r(t)^2$, $t \in \Z_{\geq 1}$, 
are not square-free (i.e., $r(t) > 1$)}. By comparison, it is well known that many polynomials, in 
the literature, give subfamilies of integers (especially fundamental units) found by 
means of the $m_{s\nu}$'s with assuming that the radical $m_{s\nu}(t)$ are square-free.

\begin{remark}\label{MB}
{\rm It is accepted and often proven that the integers $t^2 - 4 s \nu$ are 
square-free with a non-zero density and an uniform repartition 
(see, e.g., \cite{FrIw}, \cite{Rud}); so an easy 
heuristic is that the last $M = M_\BB$ of the list ${\sf VM}$ is equivalent to $\BB^2$. 
This generalizes to the \fop algorithm applied to polynomials of the form 
$a_n t^n+a_{n-1}t^{n-1} + \cdots + a_0$, $n \geq 1$, $a_n \in \Z_{\geq 1}$, and 
gives the equivalent $M_\BB \sim a_n \BB^n$ as $\BB \to \infty$.

\smallskip
The main fact is that the \fop algorithm will give {\it fundamental solutions}
of norm equations $u^2 - M v^2 = 4 s \nu$ (see Section \ref{S4}), whatever the order of 
magnitude of $r$; for 
small values of $M$, $r$ may be large, even if $r(t)$ tends to $1$ as $M(t)$ tends to its 
maximal value, equivalent to $\BB^2$, as $t \to \infty$. Otherwise, without the \fop principle, 
one must assume $m_{s \nu}(t)$ square-free in the applications, as it is often done in the literature.}
\end{remark}

\section{First examples of application of the \texorpdfstring{\fop}{Lg} algorithm} 
\label{listes}

\subsection{Kummer radicals and discriminants given by \texorpdfstring{$m_s(t)$}{Lg}}

Recall that, for $t \in \Z_{\geq 1}$, we put $m_s(t) = M(t) r(t)^2$, $M(t)$ square-free.

\subsubsection{Kummer radicals}
The following program gives, as $t$ grows from $1$ up to $\BB$, the Kummer radical
$M$ and the integer $r$ obtained from the factorizations of $m'_1(t) = t^2 - 1$, under 
the form $M r^2$;  then we put them in a list ${\sf LM}$ and the \fop algorithm gives the 
pairs ${\sf C=core(mt,1)=[M, r]}$, in the increasing order of the radicals $M$ and removes 
the duplicate entries:

\ft\begin{verbatim}
MAIN PROGRAM GIVING KUMMER RADICALS
{B=10^6;LM=List;for(t=1,B,mt=t^2-1;C=core(mt,1);L=List(C);
listput(LM,vector(2,c,L[c])));M=vecsort(vector(B,c,LM[c]),1,8);
print(M);print("#M = ",#M)}
[M,r]=
[0,1],
[2,2],[3,1],[5,4],[6,2],[7,3],[10,6],[11,3],[13,180],[14,4],[15,1],[17,8],[19,39],[21,12],
[22,42],[23,5],[26,10],[29,1820],[30,2],[31,273],[33,4],[34,6],[35,1],[37,12],[38,6],
[39,4],[41,320],[42,2],[43,531],[46,3588],[47,7],[51,7],[53,9100],[55,12],[57,20],
[58,2574],[59,69],[62,8],[65,16],[66,8],[67,5967],[69,936],[70,30],[71,413],[74,430],
(...)
[999980000099,1],[999984000063,1],[999988000035,1],[999992000015,1]
#M = 999225
\end{verbatim}\ns

\begin{remark}{\rm 
Some radicals are not found. Of course they will appear for $\BB$ larger according 
to Proposition \ref{existence}.
For instance, the Kummer radical $M=94$ depends on the fundamental unit $\varepsilon_{94} =
2143295 + 221064 \sqrt {94}$ of norm~$1$; so, using $m'_1(t)$, 
the minimal solution is $t=2143295$. For the Kummer radical $M=193$, $\varepsilon_{193} =
1764132+126985 \sqrt{193}$ is of norm $-1$ and $m'_{-1}(1764132)
= 193 \times 126985^2$. 
So $t^2-1 = 193 r^2$ has the minimal solution 
$t=6224323426849$ corresponding to $\varepsilon_{193}^2$.}
\end{remark}

\subsubsection{Discriminants}
If one needs the discriminants of the quadratic fields in the ascending order, 
it suffices to replace the Kummer radical ${\sf M=core(mt)}$ by ${\sf quaddisc(core(mt))}$ 
giving the discriminant $D$ of $\Q(\sqrt M)$. We use $m'_1(t)$ and $m'_{-1}(t)$
together to get various $M$ modulo $4$ (thus the size of the list ${\sf [D]}$ is ${\sf 2*\BB}$); 
this yields the following program and results with outputs ${\sf [D]}$:

\ft\begin{verbatim}
MAIN PROGRAM GIVING DISCRIMINANTS
{B=10^6;LD=List;for(t=1,B,L=List([quaddisc(core(t^2-1))]);
listput(LD,vector(1,c,L[c]));L=List([quaddisc(core(t^2+1))]);
listput(LD,vector(1,c,L[c])));D=vecsort(vector(2*B,c,LD[c]),1,8);
print(D);print("#D = ",#D)}
[D]=
[[0],
[5],[8],[12],[13],[17],[21],[24],[28],[29],[33],[37],[40],[41],[44],[53],[56],[57],[60],
[61],[65],[69],[73],[76],[77],[85],[88],[89],[92],[93],[97],[101],[104],[105],[113],[120],
[124],[129],[136],[137],[140],[141],[145],[149],[152],[156],[161],[165],[168],[172],
[173],[177],[184],[185],[188],[197],[201],[204],[205],[209],[213],[220],[221],[229],
(...)
[3999960000104],[3999968000060],[3999976000040],[3999992000008]
#D = 1998451
\end{verbatim}\ns

This possibility is valid for all programs of the paper; we will classify the Kummer radicals, 
instead of discriminants, because radicals are more related to norm equations, but any 
kind of output can be done easily.

\subsection{Application to minimal class numbers}
One may use this classification of Kummer radicals and compute orders 
$h$ of some invariants, then apply the \fop principle, with the instruction 
${\sf VM=vecsort(vector(B,c,LM[c]),2,8)}$ to the outputs ${\sf [M,h]}$,
to get successive possible class numbers $h$ in ascending order 
(we use here $m_1(t)=t^2-4$):

\ft\begin{verbatim}
MAIN PROGRAM GIVING SUCCESSIVE CLASS NUMBERS
{B=10^5;LM=List;for(t=3,B,M=core(t^2-4);
h=quadclassunit(quaddisc(M))[1];L=List([M,h]);
listput(LM,vector(2,c,L[c])));
VM=vecsort(vector(B-2,c,LM[c]),2,8);
print(VM);print("#VM = ",#VM)}
[M,h]=
[5,1],[15,2],[2021,3],[195,4],[4757,5],[3021,6],[11021,7],[399,8],[27221,9],[7221,10],
[95477,11],[1599,12],[145157,13],[15621,14],[50621,15],[4899,16],[267101,17],[11663,18],
(...)
[2427532899,7296],[2448270399,7356],[2340624399,7384],[1592808099,7424],[1745568399,7456],
[2443324899,7600],[2479044099,7680],[2251502499,7840],[1718102499,7968],[2381439999,8040],
[2077536399,8328],[1981140099,8384]
#VM = 2712
\end{verbatim}\ns 

One may compare using polynomials $m_s(t)$ to obtain radicals, then for instance
class numbers $h$, with the classical PARI computation:

\ft\begin{verbatim}
{B=10^6;LM=List;N=0;for(M=2,B,if(core(M)!=M,next);
N=N+1;h=quadclassunit(quaddisc(M))[1];
L=List([M,h]);listput(LM,vector(2,c,L[c])));
VM=vecsort(vector(N,c,LM[c]),2,8);
print(VM);print("#VM = ",#VM)}
[M,h]=
[[2,1],[10,2],[79,3],[82,4],[401,5],[235,6],[577,7],[226,8],[1129,9],[1111,10],[1297,11],
[730,12],[4759,13],[1534,14],[9871,15],[2305,16],[7054,17],[4954,18],[15409,19],
(...)
[78745,60],[68179,62],[57601,63],[71290,64],[87271,66],[53362,68],[56011,70],[45511,72],
[38026,74],[93619,76],[94546,80],[77779,84],[90001,87],[56170,88],[99226,94],[50626,96]]
#VM = 73
\end{verbatim}\ns 

\noindent
The lists are not comparable but are equal for ``$\BB = \infty$.''

\subsection{Application to minimal orders of \texorpdfstring{$p$}{Lg}-ramified torsion groups}
Let ${\mathfrak T}_K$ be the torsion group  of the Galois group of the maximal abelian 
$p$-ramified (i.e., unramified outside $p$ and $\infty$) pro-$p$-extension of $\Q(\sqrt M)$.
The following program, for any $p \geq 3$, gives the results by ascending 
order (outputs ${\sf [M,h=p^a,T=p^b]}$, where ${\sf h}$ is the order of the 
$p$-class group and ${\sf T}$ that of ${\mathfrak T}_K$):

\smallskip
\ft\begin{verbatim}
MAIN PROGRAM GIVING SUCCESSIVE ORDERS OF p-TORSION GROUPS 
{B=10^5;p=3;e=18;LM=List;for(t=2,B,M=core(t^2-1);
K=bnfinit(x^2-M,1);wh=valuation(K.no,p);Kt=bnrinit(K,p^e);
CKt=Kt.cyc;wt=valuation(Kt.no/CKt[1],p);L=List([M,p^wh,p^wt]);
listput(LM,vector(3,c,L[c])));VM=vecsort(vector(B-1,c,LM[c]),3,8);
print(VM);print("#VM = ",#VM)}
[M,#h_p,#T_p]=
[[3,1,1],[15,1,3],[42,1,9],[105,1,27],[1599,3,81],[1095,1,243],[23066,9,729],
[1196835,3,2187],[298662,9,6561],[12629139,27,19683],[6052830,9,59049],
[747366243,243,177147]]
#VM = 12
\end{verbatim}\ns

\subsection{Application to minimal orders of logarithmic class groups}
For the definition of the logarithmic class group $\widetilde {\mathfrak T}_p$ governing 
Greenberg's conjecture \cite{Gre1}, see \cite{Jau3,Jau4}, and for its computation, see 
\cite{BeJa} which gives the structure as abelian group. 
The following program, for $p=3$, gives the results by ascending orders
(all the structures are cyclic in this interval):

\smallskip
\ft\begin{verbatim}
MAIN PROGRAM GIVING SUCCESSIVE CLASSLOG NUMBERS
{B=10^5;LM=List;for(t=3,B,M1=core(t^2-4);M2=core(t^2+4);
K1=bnfinit(x^2-M1);Clog= bnflog(K1,3)[1];C=1;for(j=1,#Clog,
C=C*Clog[j]);L=List([M1,Clog,C]);listput(LM,vector(3,c,L[c]));
K2=bnfinit(x^2-M2);Clog= bnflog(K2,3)[1];C=1;for(j=1,#Clog,
C=C*Clog[j]);L=List([M2,Clog,C]);listput(LM,vector(3,c,L[c])));
VM=vecsort(vector(2*(B-2),c,LM[c]),3,8);
print(VM);print("#VM = ",#VM)}
[M,Clog,#Clog]=
[[5,[],1],[257,[3],3],[2917,[9],9],[26245,[27],27],[577601,[81],81],[236197,[243],243],
[19131877,[729],729],[172186885,[2187],2187],[1549681957,[6561],6561]]
#VM = 9
\end{verbatim}\ns

\section{Units \texorpdfstring{$E_s(t)$}{Lg} vs fundamental units 
\texorpdfstring{$\varepsilon_{M(t)}^{}$}{Lg}}\label{3}

\subsection{Polynomials \texorpdfstring{$m_s(t) = t^2 -4 s$}{Lg} and 
units \texorpdfstring{$E_s(t)$}{Lg}}\label{discussion}
This subsection deals with the case $\nu = 1$ about the search of quadratic units
(see also \cite[Theorem 1]{Yo}).
The polynomials $m_s(t) \in \Z[t]$ define, for $t \in \Z_{\geq 1}$, the parametrized 
units $E_s(t) = \frac{1}{2}(t+\sqrt{t^2 -4 s})$ of norm $s$ in $K := \Q(\sqrt M)$, 
where $M$ is the maximal square-free divisor of $t^2 - 4 s$. 
But $M$ is unpredictable and gives rise to the following 
discussion depending on the norm $\Sgn := \Norm(\varepsilon_M^{})$ 
of the fundamental unit $\varepsilon_M^{} =:  \frac{1}{2}(a+ b \sqrt M)$ of $K$ 
and of the integral basis of $\ZK$:

\smallskip
(i) If $s=1$, $E_1(t) = \frac{1}{2} (t+\sqrt{t^2 - 4})$ is of norm $1$; so, if $\Sgn=1$,
then $E_1(t) \in \langle \varepsilon_M^{} \rangle$,
but if $\Sgn=-1$, necessarily $E_1(t) \in \langle \varepsilon_M^2 \rangle$.

\smallskip
If $s=-1$, $E_{-1}(t) = \frac{1}{2} (t+\sqrt{t^2 + 4})$ is of norm $-1$; so, necessarily 
the Kummer radical $M$ is such that $\Sgn=-1$. 

\smallskip
(ii) If  $t$ is odd, $E_s(t)$ is written with half-integer coefficients,
$t^2-4s \equiv 1 \pmod 4$, giving $M  \equiv 1 \pmod 4$
and $\ZK = \Z\big [\frac{1+\sqrt M}{2} \big]$; so $\varepsilon_M^{}$ can not
be with integer coefficients ($a$ and $b$ are necessarily odd). 

\smallskip
If $t$ is even, $M$ may be arbitrary as well as $\varepsilon_M^{}$.

\smallskip
We can summarize these constraints by means of the following Table:
\begin{equation}\label{array}
\begin{tabular}{|l|l|l|l|c|c|c}
\hline
 \ft$t^2-4s$\ns &  
\ft$\Sgn = \Norm(\varepsilon_M^{})$\ns  & \ft$E_s(t) \in$\ns  &  
\ft$\varepsilon_M^{}=\frac{1}{2}(a+ b \sqrt M)$\ns  \\  
\hline
\ft$t^2-4$,\ \  {$t$ even} \ns& \ft$\ \ 1$\ (resp. $-1$)  & 
\ft$\langle \varepsilon_M^{} \rangle$ \ (resp. $\langle \varepsilon_M^2 \rangle$)\ns & 
\ft{$a, b$ odd or even}\ns  \\  
\hline 
\ft$t^2-4$,\ \   {$t$ odd} \ns & \ft$\ \ 1$\ (resp. $-1$) & 
\ft$\langle \varepsilon_M^{} \rangle$ \ (resp. $\langle \varepsilon_M^2 \rangle$)\ns & 
\ft{$a, b$ odd}\ns  \\  
\hline 
\ft$t^2+4$,\ \   {$t$ even}\ns & $-1$& 
\ft$\langle \varepsilon_M^{} \rangle$\ns & \ft{$a, b$ odd or even}\ns \\  
\hline 
\ft$t^2+4$,\ \   {$t$ odd}\ns & $-1$& 
\ft$\langle \varepsilon_M^{} \rangle$\ns & \ft{$a, b$ odd}\ns \\ 
\hline 
\end{tabular}
\end{equation}

\medskip
Recall that the \fop algorithm consists, after choosing the upper bound $\BB$, 
in establishing the list of {\it first occurrences}, as $t$ increases from $1$
up to $\BB$, of any square-free integer $M \geq 2$,
in the factorization $m_s(t) = M(t) r(t)^2$ (whence $M = M(t_0)$ for
some $t_0$ and $M \ne M(t)$ for all $t<t_0$), and to consider the unit:
\[E_s(t):= \hbox{$\frac{1}{2}$} \big(t + \sqrt{ t^2 - 4 s}\big) 
= \hbox{$\frac{1}{2}$} \big(t + r(t) \sqrt{M(t)}\big), \ \hbox{of norm $s$. }\]

\smallskip
The \fop is necessary since, if $t_1 > t_0$ gives the same Kummer radical $M$, 
$E_s(t_0)=\varepsilon_M^{n_0}$ and $E_s(t_1)=\varepsilon_M^{n_1}$ with $n_1 > n_0$.

\smallskip
We shall prove (Theorem \ref{fondunit}) that, under the \fop algorithm, one always 
obtains the minimal possible power $n \in \{1,2\}$ in the writing $E_s(t) = 
\varepsilon_M^n$, whence $n = 2$ if and only if $s=1$ and $\Sgn=-1$, which
means that {\it $E_s(t)$ is always the fundamental unit of norm $s$}.

\smallskip
The following result shows that any square-free integer $M \geq 2$ may be obtained for 
$\BB$ large enough.

\begin{proposition} \label{existence}
Consider the polynomial $m_1(t)= t^2 - 4$. For any square-free integer $M \geq 2$, 
there exists $t \geq 1$ such that $m_1(t) = M r^2$. 
\end{proposition}

\begin{proof} 
The corresponding equation $t^2 - 4 = M r^2$ becomes of the form $t^2 -  M r^2 = 4$.
Depending on the writing in $\Z[\sqrt M]$ ($M \equiv 2,3 \pmod 4$) or 
$\Z\big [\frac{1+\sqrt M}{2} \big]$ ($M \equiv 1 \pmod 4$),
of the powers $\varepsilon_M^n = \frac{1}{2}(t+r\sqrt M)$, $n \geq 1$, of the
fundamental unit $\varepsilon_M^{}$, this selects infinitely many $t \in \Z_{\geq 1}$.
\end{proof}

\begin{remark}\label{mprime}
{\rm One may use, instead, the polynomial $m'_1(t)= t^2 - 1$ since for any 
fundamental unit of the form $\varepsilon_M^{} = \frac{1}{2}(a+b\sqrt M)$, $a, b$ odd,
then $\varepsilon_M^3 \in \Z[\sqrt M]$, but some radicals are then obtained
with larger values of $t$; for instance, $m_1(5)=21$ and $m'_1(55)=21 \cdot 12^2$
corresponding to $55 + 12 \sqrt{21} = \big( \frac{1}{2}(5+\sqrt {21}) \big)^2$. 

\smallskip
Since for $t = 2t'$, $t^2 + 4s = 4 (t'^2 -s)$ gives the same Kummer 
radical as $t'^2 -s$, in some cases we shall use $m'_s(t) := t^2 - s$ and especially 
$m'_1(t) := t^2 - 1$ which is ``universal'' for giving all Kummer radicals.

\smallskip
With the polynomials $m_{-1}(t)= t^2 + 4$ or $m'_{-1}(t)= t^2 + 1$
a solution does exist if and only if $\Norm(\varepsilon_M^{})=-1$ and one 
obtains odd powers of $\varepsilon_M^{}$.}
\end{remark}

\subsection{Checking of the exponent \texorpdfstring{$n$}{Lg} in 
\texorpdfstring{$E_s(t)=\varepsilon_{M(t)}^n$}{Lg}}\label{verif}

The following program determines the expression of $E_s(t)$ as power of
the fundamental unit of $K$; it will find that there is no counterexample to the 
relation $E_s(t) \in \{\varepsilon_{M(t)}^{}, \varepsilon_{M(t)}^2\}$, depending on $\Sgn$, 
from Table \eqref{array}; this will be proved later (Theorem \ref{fondunit}). 
So these programs are only for verification, once for all, because they unnecessarily 
need much more execution time.

\smallskip
Since $E_s(t)$ is written in $\frac{1}{2}\Z[\sqrt M]$ and $\varepsilon_M^{}$ on
the usual $\Z$-basis of $\ZK$ denoted $\sf \{1, w\}$ by PARI (from the instruction
${\sf quadunit}$), we write $E_s(t)$ on the PARI basis ${\sf \{1, quadgen(D)\}}$,
where ${\sf D=quaddisc(M)}$ is the discriminant.

\smallskip
One must specify $\BB$ and $s$, the program takes into account the first
value ${\sf 2+s}$ of $t$ since $t=1,2$ are not suitable when $s=1$; then
the test ${\sf n>(3+s)/2}$ allows the cases $n=1$ or $2$ when $s=1$. 
The output of counterexamples is given by the (empty) list ${\sf Vn}$:

\subsubsection{Case $s=1$, $m(t)=t^2-4$ (expected exponents $n \in \{1,2\}$)}
${}$
\ft\begin{verbatim}
{B=10^6;s=1;LM=List;LN=List;for(t=2+s,B,
mt=t^2-4*s;C=core(mt,1);M=C[1];r=C[2];res=Mod(M,4);
D=quaddisc(M);w=quadgen(D);Y=quadunit(D);
if(res!=1,Z=1/2*(t+r*w));if(res==1,Z=(t-r)/2+r*w);
z=1;n=0;while(Z!=z,z=z*Y;n=n+1);L=List([M,n]);
listput(LM,vector(2,c,L[c])));
VM=vecsort(vector(B-(1+s),c,LM[c]),1,8);
print(VM);print("#VM = ",#VM);
for(k=1,#VM,n=VM[k][2];if(n>(3+s)/2,Ln=VM[k];
listput(LN,vector(2,c,Ln[c]))));Vn=vecsort(LN,1,8);
print("exceptional powers : ",Vn)}
[M,n]=
[2,2],[3,1],[5,2],[6,1],[7,1],[10,2],[11,1],[13,2],[14,1],[15,1],[17,2],[19,1],[21,1],
[22,1],[23,1],[26,2],[29,2],[30,1],[31,1],[33,1],[34,1],[35,1],[37,2],[38,1],[39,1],
[41,2],[42,1],[43,1],[46,1],[47,1],[51,1],[53,2],[55,1],[57,1],[58,2],[59,1],[61,2],
[62,1],[65,2],[66,1],[67,1],[69,1],[70,1],[71,1],[74,2],[77,1],[78,1],[79,1],[82,2],
(...)
[999982000077,1],[999986000045,1],[999990000021,1],[999997999997,1]
#VM = 998893
exceptional powers:List([])
\end{verbatim}\ns

\subsubsection{Case $s=-1$, $m(t)=t^2+4$ (expected exponents $n = 1$)}
${}$
\ft\begin{verbatim}
[M,n]=
[2,1],[5,1],[10,1],[13,1],[17,1],[26,1],[29,1],[37,1],[41,1],[53,1],[58,1],[61,1],
[65,1],[73,1],[74,1],[82,1],[85,1],[89,1],[97,1],[101,1],[106,1],[109,1],[113,1],
[122,1],[130,1],[137,1],[145,1],[149,1],[157,1],[170,1],[173,1],[181,1],[185,1],
[197,1],[202,1],[218,1],[226,1],[229,1],[233,1],[257,1],[265,1],[269,1],[274,1],
(...)
[999986000053,1],[999990000029,1],[999994000013,1],[999998000005,1]
#VM = 999874
exceptional powers:List([])
\end{verbatim}\ns

\subsection{Remarks on the use of the \texorpdfstring{\fop}{Lg}algorithm}

(i) For a matter of space, the programs do not print the units $E_s(t)$ in the outputs,
but it may be deduced easily. 
To obtain a more complete data, it suffices to replace the instructions:
\[{\sf L=List([M,n]), \ \ \  listput(LM,vector(2,c,L[c]))}, \ \ \  {\sf listput(LN,vector(2,c,Ln[c]))}\]

\noindent
by the following ones (but any information can be put in ${\sf L}$; the sole condition
being to put ${\sf M}$ as first component):
\[{\sf L=List([M,n,t]), \ \ \  listput(LM,vector(3,c,L[c]))}, \ \ \  {\sf listput(LN,vector(3,c,Ln[c]))}\]

\noindent
or simply:
\[{\sf L=List([M,t]), \ \ \ \  listput(LM,vector(2,c,L[c]))}, \ \ \ \ {\sf listput(LN,vector(2,c,Ln[c]))}\]

\noindent
giving the parameter $t$ whence the trace, then the whole integer of $\Q(\sqrt M)$; 
for instance for $m_{-1}(t)=t^2+4$ and the general program with outputs ${\sf [M,n,t]}$:

\smallskip
\ft\begin{verbatim}
[M,n,t]=
[2,1,2],[5,1,1],[10,1,6],[13,1,3],[17,1,8],[26,1,10],[29,1,5],[37,1,12],[41,1,64],
[53,1,7],[58,1,198],[61,1,39],[65,1,16],[73,1,2136],[74,1,86],[82,1,18],[85,1,9],
[89,1,1000],[97,1,11208],[101,1,20],[106,1,8010],[109,1,261],[113,1,1552],
[122,1,22],[130,1,114],[137,1,3488],...
\end{verbatim}\ns

For instance for the data ${\sf [41,1,64]}$, one has $t=64$ 
giving $t^2+4=4100$, whence the fundamental unit $E_{-1}(64)=\varepsilon_{41}^{} 
=  \frac{1}{2}(64 + 10 \sqrt{41})$. Another interesting fact is the case of ${\sf [137,1,3488]}$ 
giving a large fundamental unit at the beginning of the list.

\smallskip
(ii) The programs of \S\,\ref{verif}, computing $n$, may be used with changing 
$m_s(t)$ into other polynomials as those given Section \ref{section5}, 
or by any $T:= f(t)$ with the data ${\sf mt=T^2 \pm 4}$ and  ${\sf Z = (T+r*w)/2}$ 
as the following about units $E_s(T) = \frac{1}{2} (T + r \sqrt M)$.

\smallskip
\quad (a) ${\sf T=t^2}$ (traces are squares); all are fundamental units 
($\BB = 10^4$, outputs ${\sf [M,n]}$):

\smallskip
\ft\begin{verbatim}
{B=10^4;s=1;LN=List;LM=List;for(t=2+s,B,T=t^2;
mt=T^2-4*s;C=core(mt,1);M=C[1];r=C[2];res=Mod(M,4);
D=quaddisc(M);w=quadgen(D);Y=quadunit(D);
if(res!=1,Z=1/2*(T+r*w));if(res==1,Z=(T-r)/2+r*w);
z=1;n=0;while(Z!=z,z=z*Y;n=n+1);L=List([M,n]);
listput(LM,vector(2,c,L[c])));
VM=vecsort(vector(B-(1+s),c,LM[c]),1,8);
print(VM);print("#VM = ",#VM);
for(k=1,#VM,n=VM[k][2];
if(n!=1,Ln=VM[k];listput(LN,vector(2,c,Ln[c]))));
Vn=vecsort(LN,1,8);print("exceptional powers : ",Vn)}
[M,n]=
[7,1],[51,1],[69,1],[77,1],[187,1],[287,1],[323,1],[723,1],[1023,1],[1067,1],[1077,1],
[2397,1],[3053,1],[3173,1],[5183,1],[6347,1],[6557,1],[9799,1],[14189,1],[14637,1],
[15117,1],[16383,1],[26243,1],[29127,1],[31093,1],[39999,1],[43637,1],[47103,1],
[47213,1],[50621,1],[71111,1],[71283,1],[83517,1],[99763,1],[102613,1],[114243,1],
(...)
[9956072546774637,1],[9964048570846557,1],[9988005398920077,1],[9996000599959997,1]
#VM = 9998
exceptional powers : List([])
\end{verbatim}\ns

\quad (b) ${\sf T=prime(t)}$ (traces are prime), ${\sf s=-1}$ ($\BB = 10^4$, outputs 
${\sf [M,T=prime(t),n]}$); there is only the exception ${\sf [5,11,5]}$ 
obtained as $\varepsilon_5^5 = \frac{1}{2} (5 + 11\sqrt 5)$:

\ft\begin{verbatim}
[M,T=prime(t),n]
[5,11,5],[29,5,1],[53,7,1],[149,61,1],[173,13,1],[293,17,1],[317,89,1],[365,19,1],
[533,23,1],[773,139,1],[797,367,1],[821,16189,1],[965,31,1],[1373,37,1],[1493,2357,1],
[1685,41,1],[1781,211,1],[1853,43,1],[1997,9161,1],[2213,47,1],[2285,239,1],[2309,17539,1],
[2477,647,1],[2813,53,1],[3485,59,1],[3533,2437,1],[3653,1511,1],
(..)
[10965650093,104717,1],[10966906733,104723,1],[10968163445,104729,1]
#VM = 9995
exceptional powers:List([[5,11,5]])
\end{verbatim}\ns

\smallskip
(iii) When several polynomials $m_i(t)$, $1 \leq i \leq N$, are considered together 
(to get more Kummer radicals solutions of the problem), there is in general
commutativity of the two sequences in ${\sf for(t=1,\BB,for(i=1,N,mt=\cdots))}$ and
${\sf for(i=1,N,for(t=1,\BB,mt=\cdots))}$. But we will always use the first one.

\section{Application of the \texorpdfstring{\fop}{Lg}algorithm to norm equations}\label{S4}

We will speak of solving a norm equation in $K= \Q(\sqrt M)$, for the search of 
integers $\alpha \in \ZK^+$ such that $\Norm(\alpha) = s\nu$, for $s \in \{-1,1\}$ 
and $\nu \in \Z_{\geq 1}$ given (i.e., $\alpha = \frac{1}{2}\big(u + v \sqrt M\big)$,
$u, v \in \Z_{\geq 1}$). If the set of solutions is non-empty we will define the 
notion of {\it fundamental solution}; we will see that this definition is 
common to units ($\nu = 1$) and non-units. 

\smallskip
We explain, in Theorem \ref{unicity}, under what conditions such a fundamental 
solution for $\nu > 1$ does exist, in which case it is necessarily unique and found 
by means of the \fop, algorithm using $m_{-1}(t)$ or $m_1(t)$ (depending in 
particular on $\Sgn$).

\smallskip
Note that the resulting PARI programs only use very elementary instructions and
never the arithmetic ones defining $K$ (as ${\sf bnfinit, K.fu, bnfisintnorm, ...}$); whence 
the rapidity even for large upper bounds $\BB$.

\subsection{Main property of the trace map for units}
In the case $\nu = 1$, let $\Sgn = \Norm(\varepsilon_M^{})$; we will see 
that $\alpha$ defines the generator of the group of units of norm $s$ of $\Q(\sqrt M)$ 
when it exists (whence $\varepsilon_M^{}$ if $s=\Sgn$ or $\varepsilon_M^2$ if $\Sgn=-1$ 
and $s=1$).  

\begin{theorem}\label{thmtrace}
Let $M \geq 2$ be a square-free integer. Let $\varepsilon = \frac{1}{2}(a+b \sqrt M) >1$ 
be a unit of $K := \Q(\sqrt M)$ (non-necessarily fundamental).
Then $\Trace(\varepsilon^n)$ defines a strictly increasing sequence 
of integers for $n \geq 1$.\,\footnote{The property holds from $n=0$ ($\Trace(1)=2$), 
except for $M=5$ ($\Trace(\varepsilon_M^{}) = 1$, $\Trace(\varepsilon_M^2) = 3$) 
and $M=2$ ($\Trace(\varepsilon_M^{}) = 2$, $\Trace(\varepsilon_M^2) = 6$)).}
\end{theorem}

\begin{proof}
Set $\ov \varepsilon = \frac{1}{2}(a-b \sqrt M)$ for the conjugate of $\varepsilon$ 
and let $s = \varepsilon \ov \varepsilon  = \pm 1$
be the norm of~$\varepsilon$; then the trace of $\varepsilon^n$ is 
$T_n := \varepsilon^n + \ov \varepsilon^n = \varepsilon^n + \ffrac{s^n}{\varepsilon^n}$.
Thus, we have:
\[ \frac{T_{n+1}}{T_n} = \frac{\varepsilon^{n+1} + \ffrac{s^{n+1}}{\varepsilon^{n+1}}}
{\varepsilon^n + \ffrac{s^n}{\varepsilon^n}} =  
\frac{\varepsilon^{2(n+1)}+s^{n+1}}{\varepsilon^{n+1}} \times
\frac{\varepsilon^n}{\varepsilon^{2n}+s^n} = 
\frac{\varepsilon^{2(n+1)}+s^{n+1}}{\varepsilon^{2n+1}+s^n \varepsilon}.\]

To prove the increasing, consider $\varepsilon^{2n+1}+s^n \varepsilon$
and $\varepsilon^{2(n+1)}+s^{n+1}$, which are positive for all $n$ since 
$\varepsilon > 1$; then:
\begin{equation}\label{Depsilon}
\begin{aligned}
\Delta_n(\varepsilon)& :=  \varepsilon^{2(n+1)}+s^{n+1} - (\varepsilon^{2n+1}+s^n \varepsilon)
 = \varepsilon^{2(n+1)} - \varepsilon^{2n+1} + s^{n+1} - s^n \varepsilon \\
& = \varepsilon^{2n+1}(\varepsilon - 1) - s^n (\varepsilon - s).
\end{aligned}
\end{equation}

\smallskip
(i) Case $s=1$. Then $\Delta_n(\varepsilon) = (\varepsilon - 1) (\varepsilon^{2n+1} - 1)$
is positive.

\smallskip
(ii) Case $s=-1$. Then $\Delta_n(\varepsilon) = 
\varepsilon^{2(n+1)} - \varepsilon^{2n+1} - (-1)^n (\varepsilon + 1)$.
If $n$ is odd, the result is obvious; so, it remains to look
at the expression for $n = 2k$, $k \geq 1$:
\begin{equation} \label{Delta}
\Delta_{2k}(\varepsilon) = \varepsilon^{4k+2} - \varepsilon^{4k+1} - \varepsilon - 1.
\end{equation}

Let $f(x) := x^{4k+2} - x^{4k+1} - x - 1$; then $f'(x) = (4k+2)x^{4k+1} -  (4k+1)x^{4k} - 1$
and $f''(x) = (4k+1)x^{4k-1} [(4k+2)x - 4k] \geq 0$ for all $x \geq 1$. Thus $f'(x)$
is increasing for all $x \geq 1$ and since $f'(1)=0$, $f(x)$ is an increasing map
for all $x \geq 1$; so, for $k \geq 1$ fixed, $\Delta_{2k}(\varepsilon)$ is increasing 
regarding $\varepsilon$. 

\smallskip
Since the smallest unit $\varepsilon >1$ with positive coefficients
is $\varepsilon_0^{} := \frac{1 + \sqrt 5}{2} \approx 1.6180...$
we have to look, from \eqref{Delta}, at the map
$F(z) := \varepsilon_0^{4z+2} - \varepsilon_0^{4z+1} - \varepsilon_0^{} -1$, 
for $z \geq 1$, to check if there exists an unfavorable value of $k$; so:
\[F'(z) :=4 \log(\varepsilon_0^{}) \varepsilon_0^{4z+2} - 4 \log(\varepsilon_0^{}) \varepsilon_0^{4z+1} 
= 4 \log(\varepsilon_0^{})  \varepsilon_0^{4z+1}  (\varepsilon_0^{} - 1) >0. \]
Since $F(1) \approx 4.2360 >0$, one gets $\Delta_n(\varepsilon) >0$ in the case $s=-1$, 
$n$ even.
\end{proof}

\subsection{Unlimited lists of fundamental units of norm \texorpdfstring{$s $}{Lg},
\texorpdfstring{$s \in \{-1,1\}$}{Lg}} \label{listunits}
We have the main following result.

\begin{theorem}\label{fondunit}
Let $\BB \gg 0$ be given.
Let $m_s(t) = t^2 - 4 s$, $s \in \{-1,1\}$ fixed. Then, as $t$ grows from $1$ up to $\BB$, 
for each first occurrence of a square-free integer $M \geq 2$ in the factorization
$m_s(t) = M r^2$, the unit $E_s(t) = \frac{1}{2}(t + r \sqrt M)$ is the fundamental 
unit of norm $s$ of $\Q(\sqrt M)$ (according to the Table \eqref{array} 
in \S\,\ref{discussion}, we have $E_s(t) = \varepsilon_M^{}$ if $s = -1$ 
or if $s = \Sgn = 1$, then $E_s(t) = \varepsilon_M^2$ if $s=1$ and $\Sgn = -1$).
\end{theorem}

\begin{proof}
Let $M_0 \geq 2$ be a given square-free integer. Consider the first occurrence 
$t=t_0$ giving $m_s(t_0) = M_0 r(t_0)^2$ if it exists (existence always fulfilled 
for $s=1$ by Proposition~\ref{existence}); whence $M_0=M(t_0)$.
Suppose that $E_s(t_0) = \frac{1}{2} \big (t_0 + r(t_0) \sqrt{M(t_0)}\big)$ is not the 
fundamental unit of norm $s$, $\varepsilon_{M(t_0)}^{n_0}$ 
($n_0 \in\{1,2\}$) but a non-trivial power $(\varepsilon_{M(t_0)}^{n_0})^n$,~$n >1$. 

\smallskip
Put $\varepsilon_{M(t_0)}^{n_0}=: \frac{1}{2}(a+b  \sqrt{M(t_0)})$; 
from Table  \eqref{array}, $n_0 \in \{1,2\}$ is such that 
$\Norm(\varepsilon_{M(t_0)}^{n_0}) = s$
(recall that if $s=1$ and $\Sgn=-1$, $n_0=2$, if $\Sgn=s=1$, 
$n_0=1$; if $s=-1$, necessarily $\Sgn=-1$ and $n_0=1$, 
otherwise there were no occurrence of $M_0$ for $s=-1$ and $\Sgn=1$).

\smallskip
Then, Theorem \ref{thmtrace} on the traces implies $0 < a < t_0$. We have:
\[\hbox{$a^2 - M(t_0) b^2 = 4 s\ $ and $\ m_s(a) = a^2- 4 s =: M(a) r(a)^2$;}\]
but these relations imply $M(t_0) b^2 = M(a) r(a)^2$, whence  
$M(a)= M(t_0) = M_0$. That is to say, the pair $\big (t_0,\ M_0) \big)$ compared 
to $\big (a,\  M(a)=M_0) \big)$, was not the first occurrence of $M_0$ (absurd).
\end{proof}

\begin{corollary}\label{square}
Let $t \in \Z_{\geq 1}$ and let $E_1(t) = \frac{1}{2} \big (t + \sqrt{t^2-4}\big)$
of norm $1$. Then $E_1(t)$ is a square of a unit of norm $-1$, if and only if there exists 
$t' \in \Z_{\geq 1}$ such that $t=t'^2+2$; thus $E_1(t) = 
\big( \frac{1}{2}(t' + \sqrt{t'^2+4})\big)^2 = (E_{-1}(t'))^2$.
So, the \fop algorithm, with $m_1(t) =: M(t) r(t)^2$, gives the list of 
${\sf [M(t),t]}$ for which $\frac{1}{2} \big (t + \sqrt{t^2-4}\big) = \varepsilon_M^2$
(resp. $\varepsilon_M^{}$) if $t-2=t'^2$ (resp. if not).
\end{corollary}

\begin{corollary} \label{intersection}
Let $M \geq 2$ be a given square-free integer and consider the two lists given by the
\fop algorithm, for $m_{-1}$ and $m_1$, respectively. 
Then, assuming $\BB$ large enough, $M$ appears in the two lists if and only if 
$\Sgn=-1$. Then $t'^2+4 = M r'^2$ for $t'$ minimal gives 
the fundamental unit $\varepsilon_M^{}=\frac{1}{2}(t'+r' \sqrt M)$ and $t^2-4=M r^2$,
for $t$ minimal, gives $\varepsilon_M^2$; whence $t=t'^2+2$ and $r=r' t'$.
\end{corollary}

For $s=-1$, hence $m_{-1}(t) = t^2+4$, $t \in [1, \BB]$, we know, from 
Theorem \ref{fondunit}, that the \fop algorithm gives always the fundamental unit 
$\varepsilon_M^{}$ of $\Q(\sqrt M)$ whatever its writing in $\Z[\sqrt M]$ or in 
$\Z\big [\frac{1+\sqrt M}{2} \big]$. 

\smallskip
For $s=1$ one obtains $\varepsilon_M^2$ if 
and only if $\Sgn = -1$. So we can skip checking and use the following 
simpler program with larger upper bound $\BB=10^7$; the outputs are the 
Kummer radicals ${\sf [M]}$ in the ascending order (specify ${\sf B}$ and~${\sf s}$):

\smallskip
\ft\begin{verbatim}
MAIN PROGRAM FOR FUNDAMENTAL UNITS OF NORM s
{B=10^7;s=-1;LM=List;for(t=2+s,B,mt=t^2-4*s;
M=core(mt);L=List([M]);listput(LM,vector(1,c,L[c])));
VM=vecsort(vector(B-(1+s),c,LM[c]),1,8);
print(VM);print("#VM = ",#VM)}
s=-1
[M]=
[2],[5],[10],[13],[17],[26],[29],[37],[41],[53],[58],[61],[65],[73],[74],[82],[85],
[89],[97],[101],[106],[109],[113],[122],[130],[137],[145],[149],[157],[170],[173],
[181],[185],[193],[197],[202],[218],[226],[229],[233],[257],[265],[269],[274],[277],
[281],[290],[293],[298],[314],[317],[346],[349],[353],[362],[365],[370],[373],[389],
(...)
[99999860000053],[99999900000029],[99999940000013],[99999980000005]]
#VM = 9999742
s=1
[M]=
[2],[3],[5],[6],[7],[10],[11],[13],[14],[15],[17],[19],[21],[22],[23],[26],[29],[30],
[31],[33],[34],[35],[37],[38],[39],[41],[42],[43],[46],[47],[51],[53],[55],[57],[58],
[59],[61],[62],[65],[66],[67],[69],[70],[71],[73],[74],[77],[78],[79],[82],[83],[85],
[86],[87],[89],[91],[93],[94],[95],[101],[102],[103],[105],[107],[109],[110],[111],
(...)
[99999820000077],[99999860000045],[99999900000021],[99999979999997]
#VM = 9996610
\end{verbatim}\ns

\smallskip
The same program with outputs of the form ${\sf [M,r,t]}$ for $s=1$ gives many 
examples of squares of fundamental units.
For instance, the data ${\sf [29,5,27]}$ defines the unit $E_1(27) = \frac{1}{2} (27 + 5 \sqrt{29})$
and since $27-2=5^2$, then $t'=5$, $r'=1$ and $E_1(27) = \big(\frac{1}{2} (5 + \sqrt{29}) \big)^2 
= \varepsilon_{29}^2$.

\smallskip
Some Kummer radicals giving units $\varepsilon_M^{}$ of norm $-1$ do not appear 
up to $\BB=10^7$, e.g., $M \in$ $\{241$, $313$, $337$, $394,\,\ldots\}$; but all the 
Kummer radicals $M$, such that $\Sgn=-1$, ultimately appear as $\BB$ increases. 
So, as $\BB \to \infty$, any unit is obtained, which suggests the existence of natural 
densities in the framework of the \fop algorithm. More precisely, in the list ${\sf LM}$ 
(i.e., before using ${\sf VM=vecsort(vector(B,c,LM[c]),1,8)}$), any Kummer radical $M$ 
does appear in the list as many times as the trace of $\varepsilon_M^n$ ($n$ odd) is 
less than $\BB$, which gives for instance ($\BB=10^3$):

\ft\begin{verbatim}
[M]=
[5],[2],[13],[5],[29],[10],[53],[17],[85],[26],[5],[37],[173],[2],[229],[65],[293],[82],
[365],[101],[445],[122],[533],[145],[629],[170],[733],[197],[5],[226],[965],[257],
[1093],[290],[1229],[13],
\end{verbatim}\ns

This fact with Corollaries \ref{square}, \ref{intersection} may suggest some analytic 
computations of densities (see a forthcoming paper \cite{Gra10} for more details). For this 
purpose, we give an estimation of the gap ${\sf \order LM - \order VM = \BB - \order VM}$.

\begin{theorem} \label{gap}
Consider the \fop algorithm for units, in the interval $[1,\BB]$ and $s \in \{-1,1\}$. Let 
$\BD$ be the gap between $\BB$ and the number of results. Then, as $\BB \to \infty$:

\smallskip
(i) For the polynomial $m_{-1}(t) = t^2 + 4$, $\BD \sim \BB^\frac{1}{3}$,

\smallskip
(ii) For the polynomial $m_{1}(t) = t^2 - 4$, $\BD \sim  \BB^\frac{1}{2}$,
\end{theorem}

\begin{proof}
(i) In the list ${\sf LM}$ of Kummer radicals giving units of norm $-1$, we know, from 
Theorem \ref{fondunit}, that one obtain first the fundamental unit 
$\varepsilon_0 := \varepsilon_{M_0}^{}$
from the relation $t_0^2+4 = M_0 r_0^2$, then its odd powers $\varepsilon_{M_0}^{2n+1}$ 
for $n \in [1, n_{\max}]$ corresponding to some $t_n$ such that $t_n^2+4 = M_0 r_n^2$ and
$t_n \leq t_{\max}$ defined by the equivalence:
$$\ffrac{1}{2}\big( t_{\max}+r_{\max}\sqrt {M_0}\big) \sim
\Big(\ffrac{1}{2}\big( t_0+r_0 \sqrt {M_0}\big) \Big)^{2n_{\max}+1}$$ 
in an obvious meaning. Thus, the ``maximal unit'' is equivalent to $\BB$ giving
$$n_{\max} \sim \ffrac{1}{2}\Big[\ffrac{\log \BB}{\log t_0} -1 \Big]. $$

So, we have to estimate the sums 
$\sum_{t \in [1,\Bb]} \ffrac{1}{2} \Big[\ffrac{\log \BB}{\log t} -1 \Big]$,
where $\log(\Bb) \sim \ffrac{1}{3}\log(\BB)$.

Of course there will be repetitions in the sum, but a more precise estimation
is not necessary and we obtain an upper bound:
$$\BD \sim \sum_{t \in [1,\Bb]} \frac{1}{2}\Big[\frac{\log \BB}{\log t} -1 \Big]
\sim \log \Bb  \sum_{t \in [1,\Bb]} \frac{1}{\log t} \sim \log \Bb \cdot \frac{\Bb}{\log \Bb} 
\sim \Bb = \BB^\frac{1}{3}. $$

(ii) In the case of norm $1$, the list ${\sf LM}$ is relative to the fundamental units
of norm $1$ with all its powers (some are the squares of the fundamental units
of norm $-1$); the reasoning is the same, replacing $\frac{1}{3}$ by $\frac{1}{2}$.
\end{proof}

\subsection{Unlimited lists of fundamental integers of norm \texorpdfstring{$s \nu$, $\nu \geq 2$}{Lg}}

The \fop algo\-rithm always give lists of results, but contrary to units, some norms
$s\nu$ do not exist in a given field $K$; in other words, the \fop only give 
suitable Kummer radicals since $s\nu$ is given. Recall the well known:

\begin{theorem}\label{unicity}
Let $s \in \{-1,1\}$ and $\nu \in \Z_{\geq 2}$ be given. 

\smallskip
(i) A fundamental solution of the norm equation $u^2 - M v^2 = 4 s \nu$ 
(Definition \ref{defnu}) does exist if and only if there exists an integer principal 
ideal ${\mathfrak a}$ of absolute norm $\nu$ with a generator $\alpha \in \ZK^+$ 
whose norm is of sign $s$. 

\smallskip
Under the existence of ${\mathfrak a} = (\alpha)$, with $\Norm(\alpha) = s'\nu$,
a representative $\alpha \in \ZK^+$, modulo $\langle \varepsilon_M^{} \rangle$, does
exist whatever $s$ as soon as $\Sgn=-1$; if $\Sgn=1$, 
a fundamental solution $\alpha \in \ZK^+$ does exist if and only if $s'=s$.

\smallskip
(ii) When the above conditions are fulfilled, the fundamental solution corresponding 
to the ideal ${\mathfrak a}$ is unique (in the meaning that two generators of ${\mathfrak a}$
in $\ZK^+$, having same trace, are equal) and found by the \fop algorithm.
\end{theorem}

\begin{proof}
(i) If ${\mathfrak a} = (\alpha)$, of absolute norm $\nu$, with $\alpha = 
\frac{1}{2}(u+v \sqrt M) \in \ZK^+$, one obtains $u^2 -  M v^2 = 4s\nu$ for a 
suitable $s \in \{-1,1\}$ giving a solution with $t=u$; then $m_{s\nu}(t) = 
t^2 - 4s\nu =  M(t) r^2$, whence $M=M(u)$ and $r=v$.

\smallskip
Reciprocally, assume that the corresponding equation (in unknowns $t \geq 1$, 
$s = \pm 1$) $t^2 - 4 s\nu = M r^2$, $M \geq 2$ square-free, has a solution, whence 
$t^2 -  M r^2 = 4 s \nu$. Set $\alpha := \frac{1}{2}(t+r \sqrt M) \in \Z_K^+$; then one 
obtains the principal ideal ${\mathfrak a} = (\alpha)\ZK$ of absolute norm~$\nu$.

\smallskip
(ii) Assume that $\alpha$, $\beta$ are two generators of ${\mathfrak a}$ in $\ZK^+$ 
with common trace $t \geq 1$. Put $\beta = \alpha \cdot \varepsilon_M^n$, 
$n \in \Z$, $n \ne 0$. Then: 
\[\Trace(\beta) = \alpha \cdot \varepsilon_M^n + \alpha^\sigma \cdot \varepsilon_M^{n\sigma}
= \ffrac{\alpha^2 \cdot \varepsilon_M^{2n} + s \Sgn^n \,\nu}{\alpha \cdot \varepsilon_M^{n}},\ \ 
\  \Trace(\alpha) = \ffrac{\alpha^2 + s\,\nu}{\alpha}; \]
thus $\Trace(\beta) = \Trace(\alpha)$ is equivalent to
$\alpha^2 \cdot \varepsilon_M^{2n} + s\Sgn^n \,\nu = 
\alpha^2 \cdot \varepsilon_M^{n} + s \,\nu \varepsilon_M^{n}$, whence to:
\[ \alpha^2 \cdot \varepsilon_M^{n} (\varepsilon_M^{n} - 1) = (\varepsilon_M^{n} - \Sgn^n) s\,\nu. \]

The case $\Sgn^n = -1$ is not possible since $\Norm(\beta) = \Norm(\alpha) = 
\Norm(\varepsilon_M^n) = \Sgn^n$; so, $\Sgn^n = 1$, in which case, one gets
$\alpha^2 \cdot \varepsilon_M^{n} = s\,\nu = s\,\alpha^{1+\sigma}$, thus
$\alpha^{\sigma} = \alpha \cdot \varepsilon_M^{n}$ and $\beta = \alpha^{\sigma}$,
but in that case, $\beta \notin \ZK^+$ (absurd). Whence the unicity.
\end{proof}

\begin{remark}{\rm 
Consider the above case where $\alpha$ and $\beta$ are two generators of 
${\mathfrak a}$ in $\ZK^+$ with common trace $t \geq 1$ and norm $s\nu$.
Thus, we have seen that $\beta = \alpha^{\sigma} = \alpha \cdot \varepsilon_M^{n}$.
The ideal ${\mathfrak a} = (\alpha)$ is then invariant by $G := {\rm Gal}(K/\Q)$, so 
it is of the form ${\mathfrak a} = (q) \times \prod_{p \mid D} {\mathfrak p}^{e_p}$,
where $q \in \Z$, $D$ is the discriminant of $K$, ${\mathfrak p}^2 = p\ZK$
and $e_p \in \{0,1\}$.
In other words, we have to determine the principal ideals, products of distinct
ramified prime ideals. This is done in details in \cite[\S\S\,2.1,\,2.2]{Gra10}

\smallskip \noindent
For instance, let $M=15$ and $s\nu = -6$. One has the fundamental solution $\alpha = 3 + \sqrt{15}$ 
of norm $-6$, with the trace $t=6$; then $\alpha^{\sigma} = 3 - \sqrt{15} = \alpha \cdot (-4+\sqrt{15})
= \alpha \cdot (- \varepsilon_M^{-1})$; similarly, for $s\nu = 10$, one has the fundamental
solution $\alpha = 5 + \sqrt{15}$ of norm $10$, with trace $t = 10$ and the relation
$\alpha^{\sigma} = 5 - \sqrt{15} = \alpha \cdot (4-\sqrt{15}) = \alpha \cdot (\varepsilon_M^{\sigma})$.
These fundamental solutions are indeed given by the \fop algorithm by means of the
data ${\sf [M,t]}$:

\ft\begin{verbatim}
s.Nu=-6
[M,t]=
[1,1],[6,24],[7,2],[10,4],[15,6],...
s.Nu=10
[M,t]=
[-39,1],[-31,3],[-15,5],[-6,4],[-1,2],[1,7],[6,8],[10,20],[15,10],...
\end{verbatim}}
\end{remark}

Depending on the choice of the polynomials $m_{-1}(t)$ or $m_1(t)$, consider for 
instance, the \fop algorithm applied to $M=13$ (for which $\Sgn=-1$), $\nu=3$, gives
with $m_{-1}(t)$ the solution ${\sf [M=13,t=1]}$ whence $\alpha = \frac{1}{2}(1+ \sqrt{13})$
of norm $-3$; with $m_1(t)$ it gives ${\sf [M=13,t=5]}$, $\alpha = \frac{1}{2}(5+ \sqrt{13})$
of norm $3$; the traces $1$ and $5$ are minimal for each case.
We then compute that $\frac{1}{2}(5+ \sqrt{13}) = \frac{1}{2}(1- \sqrt{13}) (-\varepsilon_{13}^{})$.

\smallskip
But with $M=7$ (for which $\Sgn=1$), the \fop algorithm with $m_{-1}(t)$ and $\nu=3$
gives ${\sf [M=7,t=4]}$ but nothing with $m_1(t)$.

\begin{remark}{\rm 
A possible case is when there exist several principal integer ideals ${\mathfrak a}$
of absolute norm $\nu \Z$ (for instance when $\nu = q_1 q_2$ is the product of two 
distinct primes and if there exist two prime ideals ${\mathfrak q}_1$, ${\mathfrak q}_2$, 
of degree $1$, over  $q_1, q_2$, respectively, such that 
${\mathfrak a} := {\mathfrak q}_1{\mathfrak q}_2$ 
and ${\mathfrak a}' := {\mathfrak q}_1 {\mathfrak q}_2^{\sigma}$ are principal). 
Let ${\mathfrak a} =: (\alpha)$ and ${\mathfrak a}'=:(\alpha')$ of absolute norm $\nu$.
We can assume that, in each set of generators, $\alpha$ and $\alpha'$ have
minimal trace $u$ and $u'$, and necessarily we have, for instance, $u' > u$; 
since the ideals ${\mathfrak a}$ are finite in number, there exists an ``absolute''
minimal trace $u$ defining the unique fundamental solution which is that found
by the suitable \fop algorithm.

\smallskip
For instance, let $s=-1$, $\nu=15$; the \fop algorithm gives the solution
${\sf [19,4]}$, whence $\alpha = 2+\sqrt{19}$ of norm $-15$. In $K=\Q(\sqrt {19})$
we have prime ideals ${\mathfrak q}_3 = (4+\sqrt{19}) \mid 3$, 
${\mathfrak q}_5 = (9+2\sqrt{19}) \mid 5$. Then we obtain the fundamental
solution with ${\mathfrak a} = {\mathfrak q}_3^{\sigma}\,{\mathfrak q}_5$, while
${\mathfrak q}_3\,{\mathfrak q}_5 = (74+17\sqrt{19})$. The fundamental unit
is $\varepsilon_M^{} = 170+39\sqrt{19}$ of norm $\Sgn=1$ and one computes
some products $\pm \alpha \varepsilon_M^n$ giving a minimal trace with $n=-1$
and the non-fundamental solution $17+4\sqrt{19}$.}
\end{remark}

If $\nu = \prod_{q \mid \nu} q^{n_q}$, where 
$q$ denotes distinct prime numbers, there exist integer ideals ${\mathfrak a}$
of absolute norm $\nu \Z$ if and only if, for each inert 
$q \mid \nu$ then $n_q$ is even. In the \fop algorithm
this will select particular Kummer radicals $M$ for which each $q \mid \nu$,
such that $n_q$ is odd, ramifies or splits in $K=\Q(\sqrt M)$; this is
equivalent to $q \mid D$ (the discriminant of $K=\Q(\sqrt M)$) or to
$\rho_q := \big(\frac{M}{q} \big)=1$ in terms of quadratic residue symbols;
if so, we then have ideal solutions $\Norm({\mathfrak a}) = \nu \Z$.

\smallskip
Let's write, with obvious notations ${\mathfrak a} = \prd_{q,\,\\_q=0} {\mathfrak q}^{n_q} 
\prd_{q,\,\rho_q=-1} {\mathfrak q}^{2n'_q} \prd_{q,\,\rho_q=1} {\mathfrak q}^{n'_q} 
{\mathfrak q}^{n''_q \sigma}$.
Then the equation becomes $\Norm({\mathfrak a}') = \nu'\Z$ for another 
integral ideal ${\mathfrak a}'$ and another $\nu' \mid \nu$, where ${\mathfrak a}'$ 
is an integer ideal ``without any rational integer factor''. Thus, $\Norm(\alpha') 
= s\nu'$ is equivalent to ${\mathfrak a}' = \alpha' \ZK$. This depends on relations in 
the class group of $K$ and gives obstructions for some Kummer radicals $M$. 
Once a solution ${\mathfrak a}'$ principal exists (non unique) we can apply 
Theorem \ref{unicity}.

\subsubsection{Program for lists of quadratic integers of norm $\nu \geq 2$}

The program for units can be modified by choosing an integer 
$\nu \geq 2$, a sign $s \in \{-1, 1\}$ and the polynomial $m_{s\nu}(t) = t^2-4 s \nu$ 
(outputs ${\sf [M(t),t]}$):

\smallskip
\ft\begin{verbatim}
MAIN PROGRAM FOR FUNDAMENTAL INTEGERS OF NORM s.nu
{B=10^6;s=1;nu=2;LM=List;for(t=1,B,mt=t^2-4*s*nu;M=core(mt);L=List([M,t]);
listput(LM,vector(2,c,L[c])));VM=vecsort(vector(B,c,LM[c]),1,8);
print(VM);print("#VM = ",#VM)}
\end{verbatim}\ns

(i) $s=1$, $\nu=2$.
${}$
\ft\begin{verbatim}
[M,t]=
[-7,1],[-1,2],[1,3],
[2,4],[7,6],[14,8],[17,5],[23,10],[31,78],[34,12],[41,7],[46,312],[47,14],[62,16],[71,118],
[73,9],[79,18],[89,217],[94,2928],[97,69],[103,954],[113,11],[119,22],[127,4350],[137,199],
[142,24],[151,83142],[158,176],[161,13],[167,26],[191,5998],[193,56445],[194,28],
[199,255078],[206,488],[217,15],[223,30],[233,6121],[238,216],
(...)
[999986000041,999993],[999990000017,999995],[999994000001,999997],[999997999993,999999]
#VM = 999909
\end{verbatim}\ns

(ii) $s=-1$, $\nu=3$.
${}$
\ft\begin{verbatim}
[M,t]=
[1, 2],
[3,6],[7,4],[13,1],[19,8],[21,3],[31,22],[37,5],[39,12],[43,26],[57,30],[61,7],[67,16],
[73,34],[91,38],[93,9],[97,1694],[103,20],[109,73],[111,42],[127,586],[129,318],
[133,11],[139,448],[151,172],[157,50],[163,1864],[181,13],[183,54],[193,379486],
[199,28],[201,1758],[211,58],[217,766],[237,15],[241,62],[247,220],[259,32],[271,428],
(...)
[999986000061,999993],[999990000037,999995],[999994000021,999997],[999998000013,999999]
#VM = 999866
\end{verbatim}\ns

\smallskip
Consider the output ${\sf [93,9]}$ ($M = 3 \cdot 31$, $t=9$, $r=1$); then
$\alpha = A_{-3}(9) = \frac{1}{2}(9+ \sqrt {3 \cdot 31})$ of norm $-3$ with ramified prime $3$;
it is indeed the minimal solution since the equation reduces to $3x'^2+4 = 31 y^2$
with minimal $x'=3$, then minimal trace $x=9$.

\smallskip
For the output ${\sf [193,379486]}$, $\alpha = A_{-3}(379486) =
\frac{1}{2}(379486+ 27316 \sqrt {193})$ of norm $-3$; this is the minimal solution
despite of a large trace, but $\varepsilon_{193}^{} = \frac{1}{2}(1764132 + 126985 \sqrt {193})$
is very large and cannot intervene to decrease the size.

\smallskip
(iii) $s=1$, $\nu=15$.
${}$
\ft\begin{verbatim}
[M,t]=
[-59,1],[-51,3],[-35,5],[-14,2],[-11,4],[-6,6],[1,8],
[10,10],[21,9],[34,14],[61,11],[66,18],[85,20],[106,22],[109,13],[129,24],[154,26],
[165,15],[181,28],[201,312],[210,30],[229,17],[241,32],[265,1400],[274,34],[301,19],
[309,36],[346,38],[349,131],[354,414],[381,21],[385,40],[394,278],[409,41216],[421,3919],
(...)
[999982000021,999991],[999985999989,999993],[999993999949,999997],[999997999941,999999]
#VM = 999815
s=-1  nu=15
[M,t]=
[1,2],
[6,6],[10,10],[15,30],[19,4],[31,8],[34,22],[46,26],[51,12],[61,1],[69,3],[79,16],[85,5],
[94,38],[106,82],[109,7],[114,42],[115,20],[139,94],[141,9],[151,98],[159,24],[166,206],
[181,11],[186,54],[190,110],[199,536],[211,28],[214,58],[229,13],[241,52658],[249,126],
[265,130],[271,32],[274,1258],[285,15],[310,70],[331,7714],[334,146],[339,36],
(...)
[999986000109,999993],[999990000085,999995],[999994000069,999997],[999998000061,999999]
#VM = 999782
\end{verbatim}\ns

\smallskip
For instance, ${\sf [85,5]}$ illustrates Theorem \ref{unicity} with the
solution $\alpha =\frac{1}{2} (5+\sqrt{85})$ of norm $-15$, with
$(\alpha)\ZK = {\mathfrak q}_3 {\mathfrak q}_5$,
where $3$ splits in $K$ and $5$ is ramified; one verifies that the ideals
${\mathfrak q}_3$ and ${\mathfrak q}_5$ are non-principal, but their product
is of course principal. For this, one obtains the following PARI/GP verifications:

\smallskip
\ft\begin{verbatim}
k=bnfinit(x^2-85)
k.clgp=[2,[2],[[3,1;0,1]]]
idealfactor(k,3)=[[3,[0,2]~,1,1,[-1,-1]~]1],[[3,[2,2]~,1,1,[0,-1]~]1]
idealfactor(k,5)=[[5,[1,2]~,2,1,[1,2]~]2]
bnfisprincipal(k,[3,[2,2]~,1,1,[0,-1]~])=[[1]~,[1,0]~]
bnfisprincipal(k,[5,[1,2]~,2,1,[1,2]~])=[[1]~,[1,1/3]~]
A=idealmul(k,[3,[2,2]~,1,1,[0,-1]~],[5,[1,2]~,2,1,[1,2]~])
bnfisprincipal(k,A)=[[0]~,[2,-1]~]
nfbasis(x^2-85)=[1,1/2*x-1/2]
\end{verbatim}\ns

\smallskip
The data ${\sf [[0],[2,-1]]}$ gives the principality with generator 
${\sf [2,-1]}$ denoting (because of the integral basis $\{1, \frac{1}{2}x -\frac{1}{2}\}$ 
used by PARI), $2-\big[ \frac{1}{2}\sqrt{85}-\frac{1}{2} \big]= \frac{1}{2}(5-\sqrt{85}) =
\alpha^\sigma$.

\smallskip
(iv) $s=-1$, $\nu=9 \times 25$.
${}$
\ft\begin{verbatim}
[M,t]=
[1,16],
[2,30],[5,15],[10,10],[13,20],[17,120],[26,6],[29,12],[34,18],[37,5],[41,24],[53,105],
[58,70],[61,25],[65,240],[73,80],[74,42],[82,270],[85,35],[89,48],[97,1280],[101,3],
[106,54],[109,9],[113,23280],[122,330],[130,110],[137,52320],[145,360],[146,66],
[149,21],[157,55],[170,390],[173,195],[178,130],[181,27],[185,2040],
(...)
[999966001189,999983],[999978001021,999989],[999982000981,999991],[999994000909,999997]
#VM = 999448
\end{verbatim}\ns

The case ${\sf [37,5]}$ may be interpreted as follows:
$m_{-1}(5)=5^2+4\cdot 9 \cdot 25= 5^2\cdot 37$, whence
$A_{-1}(5)=\frac{1}{2}(5+ 5 \sqrt{37}) = 5 \cdot \frac{1}{2}(1+  \sqrt{37}) =: 5 B$, 
where $B := \frac{1}{2}(1+ \sqrt{37})$ is of norm $-9$ and $5$ is indeed inert in $K$.
Thus $\frac{1}{2}(1+  \sqrt{37}) \ZK$ is the square of a prime ideal 
${\mathfrak q}_3$ over $3$. 
The field $K$ is principal and we compute that
${\mathfrak q}_3 = \frac{1}{2}(5 \pm \sqrt{37}) \ZK$, ${\mathfrak q}_3^2 
= \frac{1}{2}(31 \pm 5 \sqrt{37}) \ZK$. So, $B \ZK = \frac{1}{2}(1+ \sqrt{37}) \ZK=
\frac{1}{2}(31 + 5 \sqrt{37}) \ZK$ or $\frac{1}{2}(31 - 5 \sqrt{37}) \ZK$.
We have $\varepsilon_{37}^{} = 6+ \sqrt{37}$ and we obtain that
$B=\frac{1}{2}(31- 5 \sqrt{37})\cdot \varepsilon_{37}^{}$,
showing that $\alpha = 5 \cdot \frac{1}{2}(1+  \sqrt{37})$ is the fundamental solution of 
the equation $\Norm(\alpha) = 3^2 \cdot 5^2$ with minimal trace $5$.

\smallskip
For larger integers $\nu$, fundamental solutions are obtained easily, as shown
by the following example with the prime $\nu = 1009$:

\smallskip
(v) $s=-1$, $\nu=1009$.
\ft\begin{verbatim}
[M,t]=
[2,14],[5,13],[10,102],[29,100],[37,21],[41,8],[58,42],[74,58],[101,305],[109,1617],
[113,656],[137,2504],[157,108],[173,17],[185,1168],[197,259],[202,35958],[205,33],
[209,4192],[218,854],[241,380808],[253,681],[269,620],[290,158],[313,384],[314,2090],
[317,1316],[337,6792],[341,67],[353,16496],[370,1422],[394,86742],
(...)
[999986004085,999993],[999990004061,999995],[999994004045,999997],[999998004037,999999]
#VM = 999664
\end{verbatim}\ns

\smallskip
We finish with a highly composed number $\nu$, not obvious for a calculation by hand:

\smallskip
(vi) $s=1$, $\nu=2 \times 3 \times 5 \times 7$.
${}$
\ft\begin{verbatim}
[M,t]=
[-839,1],[-831,3],[-815,5],[-791,7],[-759,9],[-719,11],[-671,13],[-615,15],[-551,17],
[-479,19],[-399,21],[-311,23],[-215,25],[-209,2],[-206,4],[-201,6],[-194,8],[-185,10],
[-174,12],[-161,14],[-146,16],[-129,18],[-111,27],[-110,20],[-89,22],[-66,24],[-41,26],
[-14,28],[1,29],[15,30],[46,32],[79,34],[114,36],[151,38],[190,40],[226,332],[231,42],
[249,33],[274,44],[319,46],[366,48],[385,35],[415,50],[466,52],[511,2758],[519,54],
[526,872],[574,56],[609,273],[610,4100],[631,58],[679,574],[681,39],[690,60],[721,511],
[751,62],[814,64],[834,636],[865,1265],[879,66],[919,2486],[946,68],[991,30158],[1009,43],
(...)
[999985999209,999993],[999989999185,999995],[999993999169,999997],[999997999161,999999]
#VM = 999715
\end{verbatim}\ns

We have not dropped the negative radicals meaning, for instance with $M=-839$,
that a solution of the norm equation does exist in $\Q(\sqrt {-839})$ with
$\alpha = \frac{1}{2}(1+ \sqrt{-839})$, or with $M=-14$ giving $\alpha = 14+\sqrt{-14}$.

\section{Universality of the polynomials \texorpdfstring{$m_{s \nu}$}{Lg}}\label{section5}
Let's begin with the following obvious result making a link with polynomials $m_{s \nu}$.

\begin{lemma}\label{trace}
Let $M \geq 2$ be a square-free integer and $K=\Q(\sqrt M)$; then, any $\alpha \in \ZK^+$
is characterized by its trace $a \in \Z$ and its norm $s \nu$, $s \in \{-1,1\}$, 
$\nu \in \Z_{\geq 1}$; from these data, $\alpha = \frac{1}{2}(a+ b \sqrt M)$
where $b$ is given by $m_{s \nu}(a) =: M b^2$.
\end{lemma}

\begin{proof}
From the equation $\alpha^2- a \alpha + s\nu=0$, we get
$\alpha = \frac{1}{2} (a + \sqrt{a^2-4 s\nu})$, where necessarily 
$a^2-4 s\nu =: M b^2$ (unicity of the Kummer radical) giving $b>0$ 
from the knowledge of $a$ and $s\nu$.
\end{proof}

\subsection{Mc Laughlin's polynomials}
Consider some polynomials that one finds in the literature; for instance that
of Mc Laughlin \cite{McL} obtained from ``polynomial continued fraction 
expansion'', giving formal units, and defined as follows.

\smallskip
Let $m \geq 2$ be a given square-free integer and let $E_m = u + v \sqrt{m}$, 
$u, v \in \Z_{\geq 1}$, be the fundamental solution of the norm equation 
(or Pell--Fermat equation)
$u^2-mv^2=1$ (thus, $E_m = \varepsilon_m^{n_0}$, $n_0 \in \{1,2,3,6\}$). 
For such $m, u, v$, each of the data below leads to the {\it fundamental 
polynomial solution} of the norm equation $U(t)^2 - m(t) V(t)^2 = 1$ (see 
\cite[Theorems 1--5]{McL}), giving the parametrized units $E_{M(t)} = U(t)+ V(t) r \sqrt {M(t)}$, 
of norm $1$ of $\Q(\sqrt {M(t)})$, where $m(t) =: M(t) r(t)^2$, $M(t)$ square-free.

\smallskip
The five polynomials $m(t)$ are:
\ft\begin{equation*}
\left\{\begin{aligned}
mcl_1(t)&=v^2 t^2 + 2 u t + m, \\
& U(t) =v^2 t+u,\  V(t) = v ; \\
mcl_2(t)&=(u - 1)^2 \big(v^2 t^2 + 2 t \big)+ m, \\
&U(t) =(u - 1) \big(v^4 t^2 + 2 v^2 t \big) + u, \ V(t) =v^3 t + v ; \\
mcl_3(t)&=(u + 1)^2 \big(v^2 t^2 + 2 t \big) + m, \\
& U(t) =(u + 1) \big(v^4 t^2 + 2 v^2 t\big) + u,\ V(t)=v^3 t + v ; \\
mcl_4(t)&=(u + 1)^2 v^2 t^2 + 2(u^2 - 1) t + m, \\
& U(t)  =\ffrac{(u + 1)^2}{u-1} v^4 t^2+2(u + 1) v^2 t + u,\  V(t) =\ffrac{u + 1}{u-1} v^3 t + v ; \\
mcl_5(t)&=(u-1)^2\big (v^6 t^4 + 4v^4 t^3 + 6 v^2 t^2 \big)  + 2(u-1)(2u-1) t + m,  \\
& U(t) =(u-1) \big( v^6 t^3 + 3 v^4 t^2 + 3 v^2 t \big) + u, \ V(t)= v^3 t +v.
\end{aligned}\right.
\end{equation*}\ns

Note that for $mcl_1(t)$ one may also use a unit $E_m = u + v \sqrt{m}$ 
of norm $-1$ since $U(t)^2 - mcl_1(t) V(t)^2 = u^2-m v^2$, which is not 
possible for the other polynomials.

\smallskip
We may enlarge the previous list with cases where the coefficients of 
$E_m$ may be half-integers defining more general units 
(as $\varepsilon_5^{}$, $\varepsilon_{13}^{}$ of norm $-1$ in the case of 
$mcl_1(t)$, then as $\varepsilon_{21}^{}$ of norm $1$ for the other $mcl(t)$). 
This will give $E_m=\varepsilon_m^{}$ or $\varepsilon_m^2$.

\smallskip
So we have the following transformation of the $mcl(t)$,
$U(t)$, $V(t)$, that we explain with $mcl_1(t)$.
The polynomial $mcl_1(t)$ fulfills the condition $U(t)^2 - mcl_1(t) V(t)^2 =
u^2 - m v^2$, which is the norm of $E_m = u + v \sqrt m$; 
so we can use any square-free integer $m \equiv 1 \pmod 4$ 
such that $E_m = \frac{1}{2} \big (u + v \sqrt m \big )$, $u, v \in \Z_{\geq 1}$ odd, and
we obtain the formal unit $E_{M(t)} = \frac{1}{2} \big (U(t) + V(t) \sqrt {mcl_1(t)} \big )$ 
under the condition $t$ even to get $U(t), V(t) \in \Z_{\geq 1}$. 
This gives the polynomials $mcl_6(t) = v^2 t^2 +2ut +m$ and the coefficients 
$U(t) = \frac{1}{2}(v^2 t + u)$, $V(t) =  \frac{1}{2}v$ of a new unit, with 
$mcl_6(t) = M(t) r(t)^2$, for all $t \geq 0$,. 

\smallskip
For the other $mcl(t)$ one applies the maps 
$t \mapsto 2t$, $t \mapsto 4t$, depending on the degrees; 
so we obtain the following list, where the
resulting unit is $E_{M(t)} = U(t) + V(t) \sqrt {m(t)}$, 
of norm $\pm 1$, under the conditions $m \equiv 1 \pmod 4$ and 
$\varepsilon_m^{} = \frac{1}{2}(u + v \sqrt m)$, $u,v$ odd:
\ft\begin{equation*}
\left\{\begin{aligned}
mcl_6(t)&=v^2 t^2 + 2u t + m, \\
& U(t) =\ffrac{1}{2}(v^2 t + u), \  V(t)= \ffrac{1}{2} v ; \\
mcl_7(t)&=(u-2)^2\big (v^2 t^2 + 2 t\big) + m, \\
& U(t) =\ffrac{1}{2} \big((u-2)(v^4 t^2 + 2v^2 t) + u \big),
\ V(t)=\ffrac{1}{2} \big(v^3 t + v \big) ; \\
mcl_8(t)&=(u + 2)^2\big( v^2 t^2 + 2 t\big) + m, \\
&U(t) =\ffrac{1}{2} \big((u + 2)( v^4 t^2 + 2v^2 t) + u \big), 
\ V(t)=\ffrac{1}{2} \big(v^3 t + v \big) ; \\
mcl_9(t)&=(u + 2)^2 v^2 t^2 + 2(u^2 - 4) t + m, \\
&U(t) =\ffrac{1}{2} \Big(\ffrac{(u + 2)^2}{u-2} v^4 t^2+2(u + 2) v^2 t + u \Big),
\ V(t)=\ffrac{1}{2} \Big(\ffrac{u + 2}{u-2} v^3 t + v\Big) ; \\
mcl_{10}(t)&=(u-2)^2 \big (v^6 t^4 + 4 v^4 t^3 + 6 v^2 t^2\big)  + 4(u-2)(u-1) t + m, \\
& U(t) =\ffrac{1}{2} \big((u-2) (v^6 t^3 + 3 v^4 t^2 + 3 v^2 t) + u \big),
\ V(t)=\ffrac{1}{2} \big(v^3 t +v \big).
\end{aligned}\right.
\end{equation*}\ns

\subsection{Application to finding units}
In fact, these numerous families of parametrized units are
nothing but the units $E_s(T) = \frac{1}{2}(T+\sqrt{T^2 - 4 s})$
when the parameter $T = U(t)$ is a given polynomial expression. This explain
that the properties of the units $E_s(T)$ are similar to that of the two
universal units $E_s(t)$, for $t \in \Z_{\geq 1}$, but, a priori, the \fop algorithm 
does not give fundamental units when $T(t)$ is not a degree $1$
monic polynomial; nevertheless it seems that the algorithm gives most 
often fundamental units, at least for all $t \gg 0$. 

\smallskip
We give the following example, using for instance the Mc Laughlin polynomial 
$mcl_{10}(t)$ with $m=301$, $u=22745$, $v=1311$, corresponding to, 
$\varepsilon_m^{} = \frac{1}{2}(22745 + 1311 \sqrt{301})$ of norm $1$
(program of Section \ref{3}); this will give enormous units 
$E_{M(t)} =: \varepsilon_{M(t)}^n$.
The output is of the form ${\sf [M(t),r(t),n]}$.
Then there is no exception to $E_{M(t)} = \varepsilon_{M(t)}^{}$ (i.e., $n=1$); 
moreover, one sees many cases of non-square-free integers $mcl_{10}(t)$:

\smallskip
\ft\begin{verbatim}
Mc LAUGHLIN UNITS
{B=10^3;LN=List;LM=List;u=22745;v=1311;for(t=1,B,
mt=(u-2)^2*(v^6*t^4+4*v^4*t^3+6*v^2*t^2)+4*(u-2)*(u-1)*t+301;
ut=1/2*((u-2)*(v^6*t^3+3*v^4*t^2+3*v^2*t)+u);vt=1/2*(v^3*t+v);
C=core(mt,1);M=C[1];r=C[2];D=quaddisc(M);w=quadgen(D);
Y=quadunit(D);res=Mod(M,4);
if(res!=1,Z=ut+r*vt*w);if(res==1,Z=ut-r*vt+2*r*vt*w);
z=1;n=0;while(Z!=z,z=z*Y;n=n+1);L=List([M,r,n]);
listput(LM,vector(3,c,L[c])));VM=vecsort(vector(B,c,LM[c]),1,8);
print(VM);print("#VM = ",#VM);
for(k=1,#VM,n=VM[k][3];if(n>1,Ln=VM[k];
listput(LN,vector(3,c,Ln[c]))));Vn=vecsort(LN,1,8);
print("exceptional powers:",Vn)}
[M,r,n]=
[656527122296918386395032242,2,1],[1594671238615711306590405613,63245,1],
[6538031892707128354912512481,1400,1],[8374054846220987469202089646,14,1],
[13294653599300065679245260247,4,1],[17461037237177260272395675419,140,1],
[28515629817043220531451663970,7672,1],[42017686932862256394245096245,1,1],
(...)
[2626102383534535069268098426753041168301,1,1]
#VM = 1000
exceptional powers : List([])
\end{verbatim}\ns

Using the Remark \ref{MB}, with $\BB=10^3$ and $m(t)$ of degree $4$, with 
leading coefficient:
$$a_4  \BB^4= (22745-2)^2 \cdot 1311^6 \cdot 10^{12}
= 2626102377422775499879732689000000000000, $$
one gets:
$$\log(2626102383534535069268098426753041168301)/\log(a_4 \BB^4)
\approx 1.000000000025...$$

\section{Non \texorpdfstring{$p$}{Lg}-rationality of quadratic fields}\label{S6}

\subsection{Recalls about \texorpdfstring{$p$}{Lg}-rationality}

Let $p \geq 2$ be a prime number. The definition of $p$-rationality of
a number field lies in the framework of abelian $p$-ramification theory.
The references we give in this article are limited to cover the subject and 
concern essentially recent papers; so the reader may look at the historical 
of the abelian $p$-ramification theory that we have given in \cite[Appendix]{Gra5},
for accurate attributions, from \v{S}afarevi\v{c}'s pioneering results, about the 
numerous approaches (class field theory, Galois cohomology, pro-$p$-group 
theory, infinitesimal theory); then use its references concerning developments of
this theory (from our Crelle's papers 1982--1983, Jaulent's infinitesimals \cite{Jau1}
(1984), Jaulent's thesis \cite{Jau2} (1986), Nguyen Quang Do's article \cite{Ng} 
(1986), Movahhedi's thesis \cite{Mov} (1988) and subsequent papers); all 
prerequisites and developments are available in our book \cite{Gra1} (2005).

\begin{definition}\label{defprat}
A number field $K$ is said to be $p$-rational if $K$ fulfills the Leopoldt conjecture at $p$
and if the torsion group ${\mathfrak T}_K$ of the Galois group of the maximal abelian 
$p$-ramified (i.e., unramified outside $p$ and $\infty$) pro-$p$-extension of $K$ is trivial.
\end{definition}

We will use the fact that, for totally real fields $K$, we have the formula:
\begin{equation}
\order {\mathfrak T}_K = \order {\mathcal C}'_K \cdot \order {\mathcal R}_K 
\cdot \order {\mathcal W}_K,
\end{equation}
where ${\mathcal C}'_K$ is a subgroup of the $p$-class group ${\mathcal C}_K$ 
and where ${\mathcal W}_K$ depends on local and global $p$-roots of unity; for 
$K=\Q(\sqrt M)$ and $p>2$, ${\mathcal C}'_K={\mathcal C}_K$ and ${\mathcal W}_K=1$ 
except if $p=3$ and $M \equiv -3 \pmod 9$, in which case ${\mathcal W}_K \simeq \Z/3\Z$. 
For $p=2$, ${\mathcal C}'_K={\mathcal C}_K$ except if $K(\sqrt 2)/K$ is unramified
(i.e., if $M=2M_1$, $M_1 \equiv 1 \pmod 4$).
Then ${\mathcal R}_K$ is the ``normalized $p$-adic regulator'' of $K$ (general definition 
for any number field in \cite[Proposition 5.2]{Gra4}). For $K=\Q(\sqrt M)$ and $p \ne 2$, 
$\order {\mathcal R}_K \sim \frac{1}{p} \log_p(\varepsilon_M^{})$; for $p=2$,
$\order {\mathcal R}_K \sim \frac{1}{2^d} \log_2(\varepsilon_M^{})$, where 
$d \in \{1,2\}$ is the number of prime ideals above $2$.

\smallskip
So $\order {\mathfrak T}_K$ is divisible by the order of ${\mathcal R}_K$, 
which gives a sufficient condition for the non-$p$-rationality of $K$.
Since ${\mathcal C}_K = {\mathcal W}_K = 1$ for $p \gg 0$, the $p$-rationality
only depends on ${\mathcal R}_K$ in almost all cases.

\begin{proposition}(\cite[Proposition 5.1]{Gra6})\label{vprk}
Let $K=\Q(\sqrt m)$ be a real quadratic field of fundamental unit $\varepsilon_m^{}$.
Let $p>2$ be a prime number with residue degree $f \in \{1,2\}$. 

(i) For $p \geq 3$ unramified in $K$, 
$v_p(\order {\mathcal R}_K) = v_p (\varepsilon_m^{p^f-1} - 1)-1$.

\smallskip
(ii) For $p > 3$ ramified in $K$, $v_p(\order {\mathcal R}_K) =
\frac{1}{2}(v_{\mathfrak p}(\varepsilon^{p-1} - 1) - 1)$, where ${\mathfrak p}^2 = (p)$.

\smallskip
(iii) For $p=3$ ramified in $K$, $v_3(\order {\mathcal R}_K)=
\frac{1}{2}(v_{\mathfrak p} (\varepsilon^6 - 1) -2-\delta)$,
where ${\mathfrak p}^2 = (3)$ and $\delta=1$ (resp. $\delta=3$) 
if $m\not\equiv -3 \pmod 9$ (resp. $m \equiv -3 \pmod 9$).
\end{proposition}

A sufficient condition for the non-triviality of ${\mathcal R}_K$ that encompasses 
all cases (since the decomposition of $p$ in $\Q(\sqrt {M(t)})$ is unpredictable
in the \fop algorithm) is $\log_p(\varepsilon_m^{}) \equiv 0 \pmod {p^2}$; this implies 
that $\varepsilon_m^{}$ is a local $p$th power at $p$. It suffices to force the 
parameter~$t$ to be such that a suitable prime-to-$p$ power of 
$E_s(t) = \frac{1}{2}\big(t + r(t) \sqrt{M(t)} \big)$ is congruent to $1$ 
modulo $p^2$. So, exceptions may arrive only when $E_s(t)$ is a 
global $p$th power.

\subsection{Remarks about \texorpdfstring{$p$}{Lg}-rationality and 
non-\texorpdfstring{$p$}{Lg}-rationality}

In some sense, the $p$-rationality of $K$ comes down to saying that the 
$p$-arithmetic of $K$ is as simple as possible and that, on the contrary,
the non $p$-rationality is the standard context, at least for some $p$ for
$K$ fixed and very common when $K$ varies in some families, for $p$ fixed.

\medskip
{\bf a)} In general, most papers intend to find $p$-rational fields, a main purpose
being to prove the existence of families of $p$-rational quadratic fields (see, e.g., 
\cite{AsBo,Ben,BGKK,BeMo,Bou,BaRa,By,Gra2,Gra6,Kop,MaRo1,MaRo2,CLS}); 
for this there are three frameworks that may exist in general, but, to simplify, we 
restrict ourselves to real quadratic fields:

\smallskip
\quad (i) The quadratic field $K$ is fixed and it is conjectured that there exist only finitely 
many primes $p>2$ for which $K$ is non $p$-rational, which is equivalent to the existence 
of finitely many $p$ for which $\frac{1}{p} \log_p(\varepsilon_K^{}) \equiv 0 \pmod p$.

\smallskip
\quad (ii) The prime $p>2$ is fixed and it is proved/conjectured that there exist 
infinitely many $p$-rational quadratic field $K$, which is equivalent to the
existence of infinitely many $K$'s for which the $p$-class group is trivial
and such that $\frac{1}{p} \log_p(\varepsilon_K^{})$ is a $p$-adic unit;
this aspect is more difficult because of the $p$-class group.

\smallskip
\quad (iii) One constructs some families of fields $K(p)$ indexed by $p$ prime.
These examples of quadratic fields often make use of Lemma \ref{trace} 
to get interesting radicals and units. 

\smallskip
For instance we have
considered in \cite[\S\,5.3]{Gra6} (as many authors), the polynomials
$t^2p^{2 \rho} +s$ for $p$-adic properties of the unit
$E = t^2p^{2 \rho} +s + t p^\rho\sqrt {t^2p^{2 \rho}+ 2s}$ of norm $1$.

Taking ``$\rho = \frac{1}{2}$, $t=1$'', one gets the unit
$E=p+s +\sqrt{p(p+2s)}$ considered in \cite{Ben} where it is proved that for 
$p>3$, the fields $\Q(\sqrt{p(p+2)})$ are $p$-rational since the $p$-class group is trivial
(for analytic reasons) and the unit $p+1+\sqrt{p(p+2)}$ is not a local $p$-power.
Note that $4 p (p+2s) = m_1(2 p+2s)$, since $\Norm(E)=1$ for all $s$. 

\smallskip
Similarly, in \cite{BeMo}, is considered the bi-quadratic fields $\Q(\sqrt{p(p+2)}, \sqrt{p(p-2)})$, 
containing the quadratic field $\Q(\sqrt{p^2-4})$ giving the unit $\frac{1}{2}(p + \sqrt{p^2-4})$
still associated to $m_1(p)$; the $p$-rationality comes from the control of the $p$-class group 
since the $p$-adic regulators are obviously $p$-adic units. 

\smallskip
Finally, in \cite{Kop}, is considered the tri-quadratic fields $\Q(\sqrt{p (p+2)}, \sqrt{p(p-2)},\sqrt{-1})$
which are proven to be $p$-rational for infinitely many primes $p$; but these fields are imaginary,
so that one has to control the $p$-class group by means of non-trivial analytic arguments.

\smallskip
The $p$-rational fields allow many existence theorems and conjectures (as the 
Greenberg's conjecture \cite{Gre2} on Galois representations with open images, 
yielding to many subsequent papers as \cite{AsBo,Ben,BeMo,Bou,BaRa,GrJa,
Jau2,Kop}); they give results in the pro-$p$-group Galois theory \cite{MaRo1}. 
Algorithmic aspects of $p$-rationality may be found in \cite{Gra3,Gra5,PiVa} and 
in \cite{BeJa} for the logarithmic class group having strong connexions with 
${\mathfrak T}_K$ in connection with another Greenberg conjecture \cite{Gre1} 
(Iwasawa's invariants $\lambda = \mu = 0$ for totally real fields); for explicit
characterizations in terms of $p$-ramification theory, see \cite{Jau4,Gra8}, 
Greenberg's conjecture being obvious when ${\mathfrak T}_K =1$.

\medskip
{\bf b)}  We observe with the following program that the polynomials:
\[\hbox{$m_s(p+1) = (p+1)^2-4 s$\ \ and\ \ $m_s(2p+2) = 4(p+1)^2-4 s$}\]
always give $p$-rational quadratic fields, apart from very rare exceptions 
(only four ones up to $10^6$) due to the fact that the units 
$E_s(p+1) = \frac{1}{2}\big (p+1+\sqrt{(p+1)^2 - 4 s} \big )$ and 
$E_s(2p+2) = p+1+\sqrt{(p+1)^2 - s}$ may be a local $p$-power 
as studied in \cite{Gra2} in a probabilistic point of view (except in the case of 
$E_1(2p+2) = 1+p+\sqrt{p^2+2p} \equiv 1 \pmod {\mathfrak p}$,
with ${\mathfrak p}^2 = (p)$, thus never local $p$th power):

\smallskip
\ft\begin{verbatim}
{nu=8;L=List([-4,-1,1,4]);for(j=1,4,d=L[j];
print("m(p)=(p+1)^2-(",d,")");forprime(p=3,10^6,
M=core((p+1)^2-d);K=bnfinit(x^2-M);
wh=valuation(K.no,p);Kmod=bnrinit(K,p^nu);
CKmod=Kmod.cyc;val=0;d=#CKmod;
for(k=1,d-1,Cl=CKmod[d-k+1];
w=valuation(Cl,p);if(w>0,val=val+w));if(val>0,
print("p=",p," M=",M," v_p(#(p-class group))=",wh,
" v_p(#(p-torsion group))=",val))))}

m(p)=(p+1)^2+4, p=13     M=2         v_p(#(p-class group))=0 
                                     v_p(#(p-torsion group))=1
m(p)=(p+1)^2+1, p=11     M=145       v_p(#(p-class group))=0 
                                     v_p(#(p-torsion group))=2
                p=16651  M=277289105 v_p(#(p-class group))=0 
                                     v_p(#(p-torsion group))=1
m(p)=(p+1)^2-1, p=3      M=15        v_p(#(p-class group))=0 
                                     v_p(#(p-torsion group))=1
m(p)=(p+1)^2-4
\end{verbatim}\ns

\smallskip
The case of $p=3$, $M=15$ does not come from the regulator, nor from the class 
group, but from the factor $\order {\mathcal W}_K = 3$ since $15 \equiv -3 \pmod 9$;
but this case must be considered as a trivial case of non-$p$-rationality.

\medskip
{\bf c)}  For real quadratic fields, the $2$-rational fields are characterized via 
a specific genus theory and are exactly the subfields of the form $\Q(\sqrt m)$ 
for $m=2$, $m=\ell$, $m=2\ell$, where $\ell$ is a prime number congruent to 
$\pm 3 \pmod 8$ (see proof and history in \cite[Examples IV.3.5.1]{Gra3}). 
So we shall not consider the case $p=2$ since the non-$2$-rational quadratic 
fields may be easily deduced, as well as fields with non-trivial $2$-class group.

\medskip
{\bf d)}  Nevertheless, these torsion groups ${\mathfrak T}_K$ are ``essentially'' the 
Tate--\v{S}afarevi\v{c} groups (see their cohomological interpretations in \cite{Ng}):
\[{\rm III}_K^2 := {\rm Ker} \Big [{\rm H}^2 ({\mathcal G}_{K,S_p},\F_p) \rightarrow
\plus_{{\mathfrak p} \in S_p} \, {\rm H}^2 ({\mathcal G}_{K_{\mathfrak p}},\F_p) \Big], \] 
where $S_p$ is the set of $p$-places of $K$,
${\mathcal G}_{K,S_p}$ the Galois group of the maximal $S_p$-ramified  
pro-$p$-extension of $K$ and ${\mathcal G}_{K_{\mathfrak p}}$ the local 
analogue over $K_{\mathfrak p}$; so their non-triviality has an important
arithmetic meaning about the arithmetic complexity of the number fields
(see for instance computational approach of this context in \cite{Gra7} for the 
pro-cyclic extension of $\Q$ and the analysis of the Greenberg's conjecture 
\cite{Gre1} in \cite{Gra8}). When the set of places $S$ does not contain
$S_p$, few things are known about ${\mathcal G}_{K,S}$; see for
instance Maire's survey \cite{Mai} and its bibliography, then
\cite[Section 3]{Gra5} for numerical computations.

\smallskip
In other words, the non-$p$-rationality (equivalent,  for $p>2$, to ${\rm III}_K^2 \ne 0$)
is an obstruction to a local-global principle and is probably more mysterious than 
$p$-rationality. Indeed, in an unsophisticated context, it is the question of the number of 
primes $p$ such that the Fermat quotient $\ffrac{2^{p-1}-1}{p}$ is divisible by $p$, 
for which only two solutions are known; then non-$p$-rationality is the same
problem applied to algebraic numbers, as units $\varepsilon_M^{}$; this aspect is 
extensively developed in \cite{Gra2} for arbitrary Galois number fields).

\subsection{Families of local \texorpdfstring{$p$}{Lg}-th power units -- 
Computation of \texorpdfstring{${\mathfrak T}_K$}{Lg}}

We shall force the non triviality of ${\mathcal R}_K$ to obtain the non-$p$-rationality
of $K$.

\subsubsection{Definitions of local \texorpdfstring{$p$}{Lg}-th power units}
Taking polynomials stemming from suitable polynomials $m_s$ we can state:

\begin{theorem}\label{prational}
Let $p>2$ be a prime number and let $s \in \{-1,1\}$. 

\smallskip
(a) Let $a \in \Z_{\geq 1}$ and $\delta \in\{1,2\}$. We consider 
$T := 2\delta^{-1}(ap^4 t^2 - \delta s)$ and $m_1(T)$ giving rise to the unit 
$E_1(T) = \frac{1}{2}\big(T+\sqrt{T^2 - 4}\big) =
\frac{1}{\delta} \big(ap^4 t^2 - \delta s + p^2 t \sqrt{a^2 p^4 t^2 - 2 \delta a s}\big )$,
of norm~$1$ and local $p$th power at $p$.

\smallskip\noindent
The cases $(a, \delta) \in \{(1, 1), (1, 2), (2, 1), (3, 1), (3, 2), (4, 1), (5, 1), (5, 2) \}$
give distinct units.

\smallskip
(b) Consider $T := t_0 + p^2 t$ and $m_s(T)= T^2-4s$ and the units of norm $s$:
\[ E_s(T) = \ffrac{1}{2}\big (T + \sqrt{T^2-4s} \big);\] 
they are, for all $t$, local $p$th power at $p$ for suitable $t_0$ depending on 
$p$ and $s$, as follows:

\smallskip
\quad (i) For $t_0=0$, the units $E_s(T) = E_s(p^2 t)$ are local $p$th powers at $p$.

\smallskip
\quad (ii) For $p \not\equiv 5 \pmod 8$, there exist $s \in \{-1,1\}$ 
and $t_0 \in \Z_{\geq 1}$ solution of the congruence $t_0^2 \equiv 2s \pmod {p^2}$ 
such that the units $E_s(T)$ are local $p$th powers at $p$.

\smallskip
\quad (iii) We get the data $(p=3, s=-1, t_0 \in \{4, 5\})$, $(p=7, s=1, t_0 \in \{10, 39\})$, 
$(p=11, s=-1, t_0 \in \{19, 102\})$, $(p=17, s = -1, t_0 \in \{24, 265\}; s=1, t_0  \in \{45, 244\})$.

\medskip
As $t$ grows from $1$ up to $\BB$, for each first 
occurrence of a square-free integer $M \geq 2$ in the factorization 
$m(t) = a^2 p^4 t^2 - 2 \delta a s = M(t) r(t)^2$ (case (a)), or the factorization
$m(t) = (t_0 + p^2 t)^2-4s = M(t) r(t)^2$ (case (b)), the quadratic fields 
$\Q(\sqrt {M(t)})$, are non $p$-rational, apart possibly when 
$\frac{1}{\delta} \big(ap^4 t^2 - \delta s + p^2 t\, r(t) \sqrt{M(t)}\big) \in 
\langle \varepsilon_{M(t)}^p\rangle$ (case (a)), or
$\ffrac{1}{2}\big (t_0 + p^2 t + \sqrt{(t_0 + p^2 t)^2-4s} \big)
\in \langle \varepsilon_{M(t)}^p\rangle$ (case (b)).
\end{theorem}

\begin{proof} The case (a) is obvious. Since the case (b)\,(i) is also obvious,
assume $t_0 \not\equiv 0 \pmod {p^2}$. We have:
\[ \big(E_s(T) \big)^2 \equiv \ffrac{1}{2} \big (T^2 - 2s + T \sqrt{T^2-4s}\, \big) \pmod {p^2}, \] 
whence $\big(E_s(T) \big)^2 \equiv \frac{t_0}{2} \sqrt{T^2-4s} \pmod {p^2}$ 
under the condition $t_0^2 \equiv 2s \pmod {p^2}$. So, $E_s(T)^4 \equiv 
\frac{1}{4} t_0^2( t_0^2-4s) \equiv -1 \pmod{p^2}$, whence the result.
One computes that $t_0^2 \equiv 2s \pmod {p^2}$ has solutions 
for $(p-1)(p+1) \equiv 0 \pmod {16}$ when $s=1$ and 
$(p-1)(p+5) \equiv 0 \pmod {16}$ when $s=-1$.
\end{proof}

For instance, in case (a) we shall use $m(t)=p^4 t^2 - s$, $m(t)=p^4 t^2 - 2s$,  
$m(t)=p^4 t^2 - 4s$, $m(t)=9 p^4 t^2 - 6s$, $m(t)=9 p^4 t^2 - 12s$, 
$m(t)=4p^4 t^2 - 2s$,  $m(t)=25 p^4 t^2 - 10s$, $m(t)=25 p^4 t^2 - 20s$.
The case (b) has the advantage that the traces of the units are in $O(t)$
instead of $O(t^2)$ for case (a).

\smallskip
Since in many computations we are testing if some unit $E_s(T)$ is a global $p$th power,
we state the following result which will be extremely useful in practice because it means 
that the exceptional cases are present only at the beginning of the \fop list:

\begin{theorem}\label{heuristic}
Let $T$ be of the form $T = c t^h + c_0$, $c \geq 1$, $h \geq 1$, $c_0 \in \Z$ 
fixed and set $T^2 - 4s = M(t) r(t)^2$ when $t$ runs through $\Z_{\geq 1}$. 
For $\BB \gg 0$, the maximal bound $M_\BB^\pow$ of the square-free integers $M(t)$, 
obtained by the \fop algorithm, for which $E_s(T) := \frac{1}{2}\big (T + \sqrt{T^2 - 4s}\big)$ 
may be a $p$th power in $\langle \varepsilon_{M(t)}^{} \rangle$
(whence the field $\Q(\sqrt {M(t)})$ being $p$-rational by exception), 
is of the order of $(c^2 \BB^{2h})^{\frac{1}{p}}$ as $\BB \to \infty$.
\end{theorem}

\begin{proof}
Put $\varepsilon_M^{} = \frac{1}{2} (a + b\sqrt M)$ as usual; then we can write
$\varepsilon_M^{} \sim b\sqrt M$ and $E_s(T) \sim T$ so that $T$ and $(b\sqrt M)^p$ 
are equivalent as $M$ and $\BB$ tend to infinity; taking the most unfavorable case
$b=1$, we conclude that $M_\BB^\pow \ll (c^2 \BB^{2h})^{{2}/{p}}$ in general.
\end{proof}

For instance $T=t_0+p^2 t$, of the case (b) of Theorem \ref{prational},
gives a bound $M_\BB^\pow$, of possible exceptional Kummer radicals, of the order 
of $(p^4 \BB^{2})^{{1}/{p}}$. This implies that when $\BB \to \infty$, the density
of Kummer radicals $M$ such that $E_s(T)$ is not a global $p$th power
is equal to $1$.
With $\BB=10^6$, often used in the programs, the bound $M_\BB^\pow$
tends to $1$ quickly as $p$ increases. In practice, for almost all primes $p$, 
the \fop lists are without any exception (only the case $p=3$ gives larger
bounds, as $M_{10^6}^\pow \approx 43267$ for the above example; but it
remains around $10^6 - 43267 = 956733$ certified solutions $M$).

\subsubsection{Program of computation of ${\mathfrak T}_K$}

In case a) of Theorem \ref{prational},
we give the program using together the $16$ parametrized radicals 
and we print short excerpts. The parameter ${\sf e}$ must be 
large enough such that ${\sf p^e}$ annihilates ${\mathfrak T}_K$; 
any prime number $p>2$ may be illustrated (here we take $p=3, 5, 7$).
A part of the program is that given in \cite{Gra3} for any number field.
For convenience, we replace a data of the form ${\sf [7784110,List([9])]}$, in the outputs,
by ${\sf [7784110,[9]]}$ giving a $3$-group ${\mathfrak T}_K$ of $\Q(\sqrt{7784110})$
isomorphic to $\Z/9\Z$.

\smallskip
\ft\begin{verbatim}
{B=10^4;p=3;Lm=List([List([1,-4]),List([1,-2]),
List([1,-1]),List([1,1]),List([1,2]),List([1,4]),
List([4,-2]),List([4,2]),List([9,-6]),List([9,6]),
List([9,-12]),List([9,12]),List([25,-10]),
List([25,10]),List([25,-20]),List([25,20])]);
e=8;p4=p^4;Ln=List;LM=List;
for(t=1,B,for(ell=1,16,a=Lm[ell][1];b=Lm[ell][2];
mt=a*t^2*p4+b;M=core(mt);K=bnfinit(x^2-M,1);
Kmod=bnrinit(K,p^e);CKmod=Kmod.cyc;
Tn=List;d=#CKmod;for(k=1,d-1,
Cl=CKmod[d-k+1];w=valuation(Cl,p);
if(w>0,listinsert(Tn,p^w,1)));L=List([M,Tn]);
listput(LM,vector(2,c,L[c]))));
VM=vecsort(vector(16*B,c,LM[c]),1,8);
print(VM);print("#VM = ",#VM);
for(k=1,#VM,T=VM[k];if(T[2]==List([]),
listput(Ln,vector(1,c,T[c]))));Vn=vecsort(Ln,1,8);
print("exceptions:",Vn)}
p=3
[M,Tn]=
[[2,[]],[3,[]],[5,[]],[6,[3]],[7,[]],[10,[]],[11,[]],[13,[]],[14,[]],[15,[3]],[21,[]],
[23,[]],[29,[9]],[33,[3]],[34,[]],[35,[]],[37,[]],[38,[]],[42,[9]],[53,[]],[55,[]],
[58,[3]],[61,[]],[62,[3]],[69,[3]],[74,[9]],[77,[3]],[78,[3]],[79,[9]],[82,[3]],
[83,[3]],[85,[3]],[87,[3]],[93,[3]],[103,[3]],[106,[3]],[109,[3]],[110,[]],[115,[]],
[122,[81]],[141,[9]],[142,[3]],[143,[]],[145,[]],[146,[]],[151,[3]],[159,[3]],[173,[3]],
(...)
[202378518245,[3]],[202419008110,[27,3]],[202459502005,[27,9]],[202459502015,[81]],
[202459502035,[81]],[202459502045,[81,3]],[202499999990,[3]],[202500000010,[3]]]
#VM = 139954
exceptions:List([[2],[3],[5],[7],[10],[11],[13],[14],[21],[23],[34],[35],[37],[38],[53],
[55],[61],[110],[115],[143],[145],[146],[205],[215],[221],[226],[227],[230],[437],[439],
[442],[445],[577],[890],[902],[905],[910],[1085],[1087],[1093],[1517],[1762],[1766],
[2605],[3595],[3605],[5605],[5615],[5645],[11005]])

p=5
[M,Tn]=
[[2,[]],[3,[]],[5,[]],[6,[]],[21,[]],[23,[]],[26,[]],[29,[]],[38,[5]],[39,[5]],[51,[5]],
[62,[25]],[69,[5]],[89,[25]],[102,[]],[107,[5]],[114,[5]],[127,[5]],[134,[5]],[161,[5]],
[183,[5]],[186,[5]],[213,[]],[219,[]],[231,[]],[237,[]],[278,[5]],[287,[5]],[295,[25]],
[326,[5]],[382,[5]],[422,[5]],[434,[25]],[453,[5]],[467,[5]],[501,[5]],[509,[5]],
[514,[25]],[519,[5]],[574,[5]],[581,[5]],[606,[125]],[623,[5]],[626,[5]],[627,[5]],
[629,[5]],[645,[5]],[662,[5]],[674,[5]],[761,[5]],
(...)
[1561562640635,[125]],[1561562640645,[25]],[1561875062510,[25]],[1562187515605,[625]],
[1562187515615,[125]],[1562187515635,[25]],[1562187515645,[625]],[1562500000010,[15625]]]
#VM = 139982
exceptions:List([[2],[3],[5],[6],[21],[23],[26],[29],[102],[213],[219],[231],[237]])

p=7
[M,Tn]=
[[6,[7]],[37,[7]],[74,[7]],[101,[7]],[123,[7]],[145,[49]],[149,[7]],[206,[7]],[214,[7]],
[215,[7]],[219,[7]],[267,[7]],[505,[7]],[554,[7]],[570,[7]],[629,[7]],[663,[7]],[741,[7]],
[817,[49]],[834,[49]],[887,[49]],[894,[7]],[1067,[7]],[1373,[49]],[1446,[7]],[1517,[7]],
[1590,[7]],[1893,[7]],[2085,[7]],[2162,[7]],[2302,[49]],[2355,[7]],[2397,[7]],[2399,[7]],
[2402,[7]],[2405,[7]],[2498,[7]],[2567,[7]],[2615,[7]],[2679,[7]],[2742,[7]],[2778,[7]],
(...)
[5998899040235,[49]],[6000099240090,[7]],[6000099240110,[7]],[6001299560005,[7]],
[6001299560015,[7]],[6001299560035,[7]],[6001299560045,[7]],[6002499999990,[7]]]
#VM = 139991
exceptions:List([])
\end{verbatim}\ns

For $p=11$ and $13$ no exception is found for ${\sf B=10^4}$.

\medskip
The case b) of Theorem \ref{prational} gives an analogous program and will be 
also illustrated in the Section \ref{sec7} about $p$-class groups, especially for the 
case $p=3$. The results are similar and give, in almost cases, non-trivial $p$-adic 
regulators ${\mathcal R}_K$, hence non-$p$-rational fields $K$:

\smallskip
\ft\begin{verbatim}
p-RATIONALITY
{B=10^4;p=3;e=8;p4=p^4;Ln=List;LM=List;
for(t=1,B,forstep(s=-1,1,2,mt=p4*t^2-4*s;M=core(mt);
K=bnfinit(x^2-M,1);Kmod=bnrinit(K,p^e);
CKmod=Kmod.cyc;Tn=List;d=#CKmod;
for(k=1,d-1,Cl=CKmod[d-k+1];w=valuation(Cl,p);
if(w>0,listinsert(Tn,p^w,1)));L=List([M,Tn]);
listput(LM,vector(2,c,L[c]))));
VM=vecsort(vector(2*B,c,LM[c]),1,8);
print(VM);print("#VM = ",#VM);
for(k=1,#VM,T=VM[k];if(T[2]==List([]),
listput(Ln,vector(1,c,T[c]))));Vn=vecsort(Ln,1,8);
print("exceptions:",Vn)}
p=3
[M,Tn]=
[[2,[]],[5,[]],[10,[]],[13,[]],[14,[]],[29,[9]],[35,[]],[37,[]],[58,[3]],[61,[]],[62,[3]],
[74,[9]],[77,[3]],[82,[3]],[85,[3]],[106,[3]],[109,[3]],[110,[]],[122,[81]],[143,[]],
[145,[]],[173,[3]],[181,[3]],[182,[9]],[202,[3]],[221,[]],[226,[]],[229,[3]],[257,[27]],
[287,[3]],[323,[3]],[359,[9]],[397,[3]],[401,[3]],[410,[27]],[437,[]],[442,[]],[445,[]],
[506,[9]],[515,[3]],[518,[3]],[533,[3]],[626,[3]],[635,[3]],[674,[9]],[730,[27]],
(...)
[8078953685,[81]],[8078953693,[9]],[8082189805,[3,3]],[8085426557,[9]],
[8085426565,[9]],[8088663965,[9]],[8091902021,[3]],[8095140733,[3]],
[8098380077,[27,3]],[8098380085,[27]]
#VM = 19990
exceptions:List([[2],[5],[10],[13],[14],[35],[37],[61],[110],[143],[145],[221],[226],
[437],[442],[445],[1085],[1093],[1517]])

p=5
[M,Tn]=
[[6,[]],[21,[]],[26,[]],[29,[]],[39,[5]],[51,[5]],[69,[5]],[89,[25]],[114,[5]],[161,[5]],
[326,[5]],[434,[25]],[501,[5]],[509,[5]],[514,[25]],[574,[5]],[581,[5]],[626,[5]],
[629,[5]],[674,[5]],[761,[5]],[789,[5]],[791,[5]],[874,[5]],[1086,[5]],[1111,[5,5]],
[1191,[5]],[1351,[5]],[1406,[5]],[1641,[625]],[1761,[5]],[1851,[5]],[1914,[5]],
(...)
[62412530621,[5]],[62412530629,[5]],[62437515621,[25]],[62437515629,[125]], 
[62462505621,[5]],[62462505629,[5]],[62487500621,[5]],[62487500629,[5]]]
#VM = 19996
exceptions:List([[6], [21], [26], [29]])

p=7
[M,Tn]=
[[6,[7]],[37,[7]],[101,[7]],[145,[49]],[149,[7]],[206,[7]],[215,[7]],[554,[7]],[570,[7]],
[629,[7]],[663,[7]],[741,[7]],[817,[49]],[894,[7]],[1067,[7]],[1373,[49]],[1517,[7]],
[1893,[7]],[2085,[7]],[2162,[7]],[2302,[49]],[2355,[7]],[2397,[7]],[2402,[7]],[2405,[7]],
[2498,[7]],[2567,[7]],[2679,[7]],[2742,[7]],[2845,[7]],[2915,[49]],[3162,[7]],[3477,[7]],
(...)
[239668014477,[7]], [239668014485,[7]],[239763977645,[343]],[239763977653,[7]],
[239859960029,[7]],[239955961613,[7]],[240051982397,[7]],[240051982405,[7]]]
#VM = 19998
exceptions:List([])
\end{verbatim}\ns

\subsection{Infiniteness of non \texorpdfstring{$p$}{Lg}-rational real quadratic fields}
All these experiments raise the question of the infiniteness, for any given prime $p \geq 22$,
of non $p$-rational real quadratic fields when the non $p$-rationality is due to 
${\mathcal R}_K \equiv 0 \pmod p$ (i.e., $\log(\varepsilon_M^{}) \equiv 0 \pmod {p^2}$). 
The case $p=2$ being trivial because of genus theory for $2$-class groups, we suppose
$p>2$. However, it is easy to prove this fact for $p=2$ by means of the regulators.

\subsubsection{Explicit families of units}
We will built parametrized Kummer radicals and units, in the corresponding fields,
which are not $p$th power of a unit; the method relies on the choice of
suitable values of the parameter trace $t$. This will imply the infiniteness of 
degree $p-1$ imaginary cyclic fields of the Section \ref{sec7} having non trivial 
$p$-class group.

\begin{theorem}\label{infty}
(i) Let $q \equiv 1 \pmod p$ be prime, let $\ov c \notin \F_q^{\times p}$  
and $t_q \in \Z_{\geq 1}$ such that $t_q \equiv \ffrac{c^2 + s}{2 c p^2} \pmod q$. 
Then, whatever the bound $\BB$, the \fop algorithm applied to the polynomial
$m(t_q + q x)
= p^4 (t_q + q x)^2-s$, $x \in \Z_{\geq 0}$, gives lists of distinct Kummer 
radicals $M$, in the ascending order, such that $\Q(\sqrt M)$ is non-$p$-rational.

\smallskip
(ii) For any given prime $p>2$ there exist infinitely many real quadratic fields $K$ 
such that ${\mathcal R}_K \equiv 0 \pmod p$, whence infinitely many non $p$-rational 
real quadratic fields. 
\end{theorem}

\begin{proof}
(i) {\it Criterion of non $p$th power.}
Consider $m(t)= p^4 t^2 -s$ and the unit $E_s(2p^2t)=p^2t + \sqrt{p^4 t^2-s}$ of 
norm $s$ and local $p$th power at $p$.
Choose a prime $q \equiv 1 \pmod p$ and let $c \in \Z_{>1}$ be non 
$p$th power modulo $q$ (whence $(q-1)\big (1-\frac{1}{p} \big)$ possibilities).  
Let $t \equiv \ffrac{c^2 + s}{2 c p^2} \pmod q$; then: 
\begin{equation*}
\begin{aligned}
\Norm(E_s(2p^2t) - c)& = \Norm(p^2 t - c + \sqrt{p^4 t^2-s}) \\
& = (p^2t-c)^2-p^4 t^2+ s =  c^2 + s - 2cp^2 t  \equiv 0 \pmod q.
\end{aligned}
\end{equation*}

Such value of $t$ defines the field $\Q(\sqrt {M(t)})$, via $p^4 t^2 -s = M(t) r(t)^2$,
and whatever its residue field at $q$ ($\F_q$ or $\F_{q^2}$), we get
$E_s(2p^2t) \equiv c \pmod {\mathfrak q}$, for some ${\mathfrak q} \mid q \Z$; since
in the inert case, $\order \F_{q^2}^\times = (q-1)(q+1)$, with $q+1 \not\equiv 0 \pmod p$,
$c$ is still non $p$th power, and $E_s(2p^2t)$ is not a local $p$th power modulo ${\mathfrak q}$,
hence not a global $p$th power.

\smallskip
(ii) {\it Infiniteness.}
Now, for simplicity to prove the infiniteness, we restrict ourselves to the case 
$m(t)=p^4 t^2 -1$ (the case $m(t)=p^4 t^2 +1$ may be considered with a similar 
reasoning in $\Z[\sqrt{-1}]$ instead of $\Z$).
Let $\ell$ be a prime number arbitrary large and consider the congruence:
$$p^2 (t_q+ q x) \equiv 1 \pmod \ell ; $$ 
it is equivalent to $x = x_0 + y \ell$, $y \in \Z_{\geq 0}$, where $x_0$ 
is a residue modulo $\ell$ of the constant $\ffrac{1-t_q p^2}{qp^2}$; so, we have
$p^2 (t_q+ q x_0) - 1 = \lambda \ell^n$, $n \geq 1$, $\ell \nmid \lambda$. 
Computing these $m(t)$'s, with $t = t_q+ (x_0 + y \ell) q$, gives:
$$p^4 (t_q+ q (x_0 + y  \ell))^2-1= [p^2 (t_q+ q (x_0 + y  \ell))-1]\cdot
[p^2 (t_q+ q (x_0 + y  \ell))+1] \equiv 0 \!\!\! \pmod \ell; $$ 
the right factor is prime to $\ell$; the left one is of the form $\lambda \ell^n + q y p^2 \ell$, 
and whatever $n$, it is possible to choose $y$ such that the $\ell$-valuation
of $\lambda \ell^{n-1} + q y  p^2$ is zero. So, for such integers $t$, we have the 
factorization $m(t) = \ell M' r^2$, where $M' \geq 1$ is square-free and $M' r^2$
prime to $\ell$, which defines $M:=\ell M'$ arbitrary large.

\smallskip
This proves that in the \fop algorithm, when $\BB \to \infty$, one can find arbitrary large 
Kummer radicals $M(t_q+ (x_0 + y \ell) q)$ such that the corresponding unit 
$E_1(t_q+ (x_0 + y \ell) q)$ is a local $p$th power modulo $p$, but not a global 
$p$-th power.
\end{proof}

The main property of the \fop algorithm is that the Kummer radicals obtained are 
distinct and listed in the ascending order; without the \fop process, all the integers 
$t = t_q+ (x_0 + y \ell) q$ giving the same $M$ give  
$E_1(t_q+ (x_0 + y \ell) q) = \varepsilon_{M}^n$ with $n \not\equiv 0 \pmod p$.

\subsubsection{Unlimited lists of non-\texorpdfstring{$p$}{Lg}-rational real 
quadratic fields}\label{3ex}
Take $p=3$, $q=7$, $c \in \{2, 3, 4, 5\}$. With $m(t)=81 t^2-1$, then $t_q \in \{2, 5\}$;
with $m(t)=81 t^2+1$, then $t_q \in  \{3, 4\}$ and $t=t_q + 7 x$, $x \geq 0$. 
The \fop list is without any exception, giving non $3$-rational quadratic fields $\Q(\sqrt M)$
(in the first case, $p=3$ is inert and in the second one, $p=3$ splits.
We give the corresponding list using together the four possibilities:

\smallskip
\ft\begin{verbatim}
NON p-RATIONAL REAL QUADRATIC FIELDS I
{B=10^6;p=3;Lm=List([List([-1,3]),List([-1,4]),
List([1,2]),List([1,5])]);Ln=List;LM=List;
for(t=1,B,for(ell=1,4,s=Lm[ell][1];tq=Lm[ell][2];
M=core(81*(tq+7*t)^2-s);L=List([M]);
listput(LM,vector(1,c,L[c]))));
VM=vecsort(vector(4*B,c,LM[c]),1,8);
print(VM);print("#VM = ",#VM)}
[M]=
[58],[74],[106],[113],[137],[359],[386],[401],[410],[494],[515],[610],[674],[743],[806],
[842],[877],[1009],[1010],[1157],[1367],[1430],[1901],[1934],[2006],[2153],[2255],
[2522],[2678],[2822],[2986],[3014],[5266],[5513],[6626],[6707],[6722],[6890],[7310],
[7610],[7858],[7919],[8101],[8465],[8555],[8738],[8761],[9410],[9634],[9998],[11183],
[11195],[11237],[11447],[11509],[11537],[11663],[11890],[11965],[13427],[13645],
[14795],[16895],[16913],[17266],[18530],[19223],[19826],[20066],[20735],[21023],
[21317],[21389],[22730],[23066],[23102],[23410],[23626],[23783],[23963],
(...)
[248061933000323],[248063067000323],[248063350500730],[248063634001297]
#VM = 4000000
\end{verbatim}\ns

\smallskip
In case of doubt about the results, one may use the same program with the computation
of $\order {\mathfrak T}_K$; but the execution time is much larger and it is not possible
to take a large $\BB$ since the computations need the instructions ${\sf K=bnfinit(x^2-M)}$
and ${\sf Kmod=bnrinit(K,p^e)}$ of class field theory package (the list below contains $42$ 
outputs up to $M=23963$, while the first one contains $80$ Kummer radicals):

\smallskip
\ft\begin{verbatim}
NON p-RATIONAL REAL QUADRATIC FIELDS II
{B=10^3;p=3;Lm=List([List([-1,3]),List([-1,4]),
List([1,2]),List([1,5])]);e=8;p4=p^4;Ln=List;LM=List;
for(t=1,B,for(ell=1,4,s=Lm[ell][1];t0=Lm[ell][2];
M=core(81*(t0+7*t)^2-s);K=bnfinit(x^2-M);
Kmod=bnrinit(K,p^e);CKmod=Kmod.cyc;
Tn=List;d=#CKmod;for(k=1,d-1,
Cl=CKmod[d-k+1];w=valuation(Cl,p);
if(w>0,listinsert(Tn,p^w,1)));L=List([M,Tn]);
listput(LM,vector(2,c,L[c]))));
VM=vecsort(vector(4*B,c,LM[c]),1,8);
print(VM);print("#VM = ",#VM);
for(k=1,#VM,T=VM[k];if(T[2]==List([]),
listput(Ln,vector(1,c,T[c]))));Vn=vecsort(Ln,1,8);
print("exceptions:",Vn)}
[M,Tn]=
[[58,[3]],[74,[9]],[106,[3]],[359,[9]],[401,[3]],[410,[27]],[515,[3]],[674,[9]],[842,[9]],
[1009,[9]],[1157,[3]],[1367,[9]],[1430,[9]],[1934,[3]],[2255,[3]],[2678,[9]],[2822,[9]],
[3014,[3]],[5513,[9]],[6722,[27]],[6890,[3]],[7310,[3,3]],[7858,[9]],[7919,[3]],[8101,[3]],
[8465,[27]],[8555,[27]],[8738,[3]],[8761,[81]],[9410,[9]],[9634,[27,3]],[9998,[9,3]],
(...)
[3955403663,[27]],[3956535802,[3]],[3957668101,[27,3]],[3965598730,[3]],
[3966732323,[9,3]],[3971268323,[81]],[3972402730,[3]],[3973537297,[27]]]
#VM = 4000
exceptions : List([])
\end{verbatim}\ns

\section{Application to \texorpdfstring{$p$}{Lg}-class groups of some imaginary cyclic fields}\label{sec7}

Considering, now, the case b) of Theorem \ref{prational} for $p>2$, we use the polynomial 
$m_s(T)= T^2-4s$, with $T=t_0+p^2 t$, and the unit of norm $s$:
\[ E_s(T) = \ffrac{1}{2}\big (T + \sqrt{T^2-4s} \big), \] 

\noindent
for suitable $s$ and $t_0$ such that $E_s(T)$ be a local $p$th power at $p$,
which is in particular the case for all $p>2$ and all $s$ when $t_0 = 0$.
For $t_0 \ne 0$, we get the particular data when the equation $t_0^2 \equiv 2s \pmod {p^2}$ 
has solutions (which is equivalent to $p \not \equiv 5 \pmod{8}$):
$(p=3, s \in \{-1,1\}, t_0 = 0)$, $(p=3, s=-1, t_0 \in \{4, 5\})$, $(p=7, s=1, t_0 \in \{10, 39\})$, 
$(p=11, s=-1, t_0 \in \{19, 102\})$, $(p=17, s = -1, t_0 \in \{24, 265\}; s=1, t_0  \in \{45, 244\})$.
For $p=2$, a ``mirror field'' may be taken in $\Q(\sqrt {-1}, \sqrt M)$ (see, e.g., \cite{Gra9} for some
results linking $2$-class groups and norms of units).

\smallskip
The programs are testing that $E_s(T)$ is not the $p$th power in $\langle \varepsilon_M^{} \rangle$.

\subsection{Imaginary quadratic fields with non-trivial \texorpdfstring{$3$}{Lg}-class group}

From the above, we obtain, as consequence, the following selection of illustrations
(see Theorem \ref{infty} claiming that the \fop lists are unbounded as $\BB \to \infty$):

\begin{theorem}\label{cubic}
Let  $t_0 \in \{0,4,5\}$ and $m(t) := (t_0+ 9 t)^2+4$ if $t_0 \ne 0$, 
or $m(t) := (t_0+ 9 t)^2 \pm 4$ if $t_0 = 0$.
As $t$ grows from $1$ up to $\BB$, each first occurrence of a square-free integer 
$M \geq 2$ in the factorization $m(t) =: M r^2$, the quadratic field 
$F_{3,M} := \Q(\sqrt {-3 M})$ has a class number divisible by $3$, except possibly
when the unit $E_s(t_0+ 9 t):=\frac{1}{2} (t_0+ 9 t+ r \sqrt M)$ is a third power 
in $\langle \varepsilon_M^{} \rangle$. 

\smallskip\noindent
The \fop algorithm applied to the subset of parameters $t=2+7 x$ or $t=5+7 x$, 
$x \in \Z_{\geq 0}$ with $m(t)=81 t^2 - 1$, always gives non-trivial $3$-class groups.
Same results with $t= \pm 3 +7 x$ with $m(t)=81 t^2 + 1$.
\end{theorem}

\begin{proof}
If $E_s(t_0+ 9 t)$ is not a third power in $\langle \varepsilon_M^{} \rangle$ but a local $3$th power 
at $3$, it is $3$-primary in the meaning that if $\zeta_3$ is a primitive $3$th root 
of unity, then $K(\zeta_3, \sqrt[3]{E_s((t_0+ 9 t))}/K(\zeta_3)$ is unramified (in fact $3$ splits 
in this extension). From reflection theorem (Scholz's Theorem in the present case), 
$3$ divides the class number of $\Q(\sqrt{-3 M})$, even when $r > 1$ in the factorization
$m(t) =: M r^2$. The case of $t_0=0$ and $s = \pm 1$ is obvious.
The second claim comes from Theorem \ref{infty} (see numerical part below).
\end{proof}

\subsubsection{Program for lists of $3$-class groups of imaginary quadratic fields}

Note that the case where $E_s(t_0+ 9 t)$ is a third power is very rare because it 
happens only for very large $t_0+ 9 t$ giving a small Kummer radical $M$.
One may verify the claim by means of the following program,
in the case $s=-1$ valid for all $t_0$, where
${\sf [M,Vh]}$ gives in ${\sf Vh}$ the $3$-structure of the class 
group of $\Q(\sqrt {-3 M})$; at the end of each output, one sees the list of 
exceptions (case of third powers), where the output ${\sf [M,n]}$ means
that for the Kummer radical $M=M(t)$, then $E_{-1}(t_0+ 9 t)=\varepsilon_M^n$. 
We may see that any excerpt for $t$ large enough give no exceptions:

\smallskip
\ft\begin{verbatim}
LISTS OF 3-CLASS GROUPS OF IMAGINARY QUADRATIC FIELDS
{p=3;B=10^5;L3=List;Lh=List;Lt0=List([0,4,5]);for(t=1,B,
for(ell=1,3,t0=Lt0[ell];mt=(t0+9*t)^2+4;ut=(t0+9*t)/2;vt=1/2;
C=core(mt,1);M=C[1];r=C[2];res=Mod(M,4);D=quaddisc(M);w=quadgen(D);
Y=quadunit(D);if(res!=1,Z=ut+r*vt*w);if(res==1,Z=ut-r*vt+2*r*vt*w);
z=1;n=0;while(Z!=z,z=z*Y;n=n+1);C3=List;K=bnfinit(x^2+3*M,1);
CK=K.cyc;d=#CK;for(j=1,d,Cl=CK[d-j+1];val=valuation(Cl,3);
if(val>0,listinsert(C3,3^val,1)));L=List([M,C3,n]);
listput(Lh,vector(3,c,L[c]))));Vh=vecsort(vector(3*B,c,Lh[c]),1,8);
print(Vh);print("#Vh = ",#Vh);
for(k=1,#Vh,LC=Vh[k][2];if(LC==List([]),Ln=List([Vh[k][1],Vh[k][3]]);
listput(L3,vector(2,c,Ln[c]))));V3=vecsort(L3,1,8);
print("exceptional powers : ",V3)}
[M,C3,n]=
[[2,[],15],[5,[],9],[10,[],3],[13,[],3],[17,[],3],[26,[],3],[29,[3],5],[37,[],3],
[41,[],3],[53,[],3],[58,[3],1],[61,[],3],[65,[],3],[74,[3],1],[82,[3],1],[85,[3],1],
[101,[],3],[106,[3],1],[109,[3],1],[113,[3],1],[122,[3],1],[137,[3],1],[145,[],3],
[149,[],3],[170,[],3],[173,[9],1],[181,[3],1],[197,[],3],[202,[3],1],[226,[],3],
[229,[3],3],[257,[3],1],[290,[],3],[293,[],3],[314,[3],1],[317,[],3],[353,[3],1],
[362,[],3],[365,[],3],[397,[3],1],[401,[3],1],[442,[],3],[445,[],3],[461,[9],1],
[485,[],3],[530,[],3],[533,[9],1],[577,[],3],[610,[3],1],[626,[3],1],[629,[],3],
[653,[3],1],[677,[],3],[730,[3],1],[733,[9],1],[754,[3],1],[773,[3],1],[785,[3],3],
[842,[3],1],[877,[3],1],[901,[],3],[962,[],3],[965,[9],1],[997,[3],1],[1009,[3],1],
(...)
[809976600173,[27],1],[809983800085,[81],1],[809991000029,[27],1],[810009000029,[9],1]]
#Vh = 299963
exceptional powers:List([[2,15],[5,9],[10,3],[13,3],[17,3],[26,3],[37,3],[41,3],[53,3],
[61,3],[65,3],[101,3],[145,3],[149,3],[170,3],[197,3],[226,3],[290,3],[293,3],[317,3],
[362,3],[365,3],[442,3],[445,3],[485,3],[530,3],[577,3],[629,3],[677,3],[901,3],[962,3],
[1093,3],[1226,3],[1370,3],[1601,3],[1853,3],[2117,3],[2305,3],[2605,3],[2813,3],
[3029,3],[3253,3],[4229,3],[5045,3],[6245,3],[6893,3],[8653,3]])
\end{verbatim}\ns

Then $M_\BB = 810016200085$ and $\log(810016200085)/\log(81\cdot 10^{10})
\approx 1.0000007293$; then $M_\BB^{\frac{1}{3}} \approx 9321.76$ give a good
verification of the Heuristic \ref{heuristic}. This also means that {\it all the integers $M$
larger than $9029$} leads to non-trivial $3$-class groups, and they are very numerous !

\smallskip
We note that some $M$'s (as $29$, $74$, $82$, $85$,\,$\ldots$) are in the list of 
exceptions despite a non-trivial $3$-class group; this is equivalent to the fact that, 
even if $E_{-1}(t_0+ 9 t) \in \langle \varepsilon_M^3 \rangle$, either the $3$-regulator 
${\mathcal R}_K$ of $K$ is non-trivial or its $3$-class group is non-trivial.

\subsubsection{Unlimited lists of non-trivial $3$-class groups}
To finish, let's give the case where the \fop algorithm {\it always} gives a non-trivial
$3$-class group in $\Q(\sqrt {-3M})$; we use together the $4$ parametrizations 
given by Theorem \ref{cubic} (outputs ${\sf [M, [3\, class\ group]]}$):

\smallskip
\ft\begin{verbatim}
NON TRIVIAL 3-CLASS GROUPS OF IMAGINARY QUADRATIC FIELDS
{p=3;B=10^4;Lh=List;Lm=List([List([-1,3]),List([-1,4]),List([1,2]),List([1,5])]);
for(t=1,B,for(ell=1,4,s=Lm[ell][1];t0=Lm[ell][2];M=core(81*(t0+7*t)^2-s);C3=List;
K=bnfinit(x^2+3*M);CK=K.cyc;d=#CK;for(j=1,d,Cl=CK[d-j+1];
val=valuation(Cl,3);if(val>0,listinsert(C3,3^val,1)));L=List([M,C3]);
listput(Lh,vector(2,c,L[c]))));Vh=vecsort(vector(4*B,c,Lh[c]),1,8);
print(Vh);print("#Vh = ",#Vh)}
[M,C3]=
[[58,[3]],[74,[3]],[106,[3]],[359,[3]],[386,[3]],[401,[3]],[410,[3]],[494,[3]],[515,[3]],
[610,[3]],[674,[3]],[842,[3]],[877,[3]],[1009,[3]],[1157,[3]],[1367,[3]],[1430,[3]],
[1901,[9,3]],[1934,[9]],[2153,[3]],[2255,[3]],[2678,[9]],[2822,[3]],[2986,[3]],[3014,[3]],
[5266,[3]],[5513,[3]],[6626,[9]],[6707,[3]],[6722,[3]],[6890,[3]],[7310,[3,3]],[7858,[27]],
[7919,[3]],[8101,[3]],[8465,[9]],[8555,[9]],[8738,[3]],[8761,[9]],[9410,[3]],[9634,[9,3]],
[9998,[3,3]],[11183,[3]],[11237,[3]],[11447,[3]],[11509,[27]],[11537,[3]],[11663,[3,3]],
[11965,[3]],[13427,[3]],[16895,[3]],[16913,[3,3]],[17266,[9]],[18530,[3]],[20066,[3]],
(...)
[396877320323,[3]],[396922680323,[9]],[396934020730,[3]],[396945361297,[3]]]
#Vh = 40000
\end{verbatim}\ns

\subsection{Imaginary cyclic fields with non-trivial \texorpdfstring{$p$}{Lg}-class 
group, \texorpdfstring{$p>3$}{Lg}}

Let $\chi$ be the even character of order $2$ defining $K := \Q(\sqrt M)$, let $p \geq 3$
and let $L:=K(\zeta_p)$ be the field obtained by adjunction of a primitive $p$th root of
unity; we may assume that $K \cap \Q(\zeta_p)=\Q$, otherwise $M=p$
in the case $p \equiv 1 \pmod 4$, case 
for which there is no known examples of $p$-primary fundamental unit. 
Let $\omega$ be the $p$-adic Teichm\"uller character (so that for all 
$\tau \in {\rm Gal}(L/\Q)$, $\zeta_p^\tau = \zeta_p^{\omega(\tau)}$). 

\smallskip
Then, for any list of quadratic fields $\Q(\sqrt M)$ obtained by the previous \fop
algorithm giving $p$-primary units $E$, the $\omega \chi^{-1}$-component 
of the $p$-class group of $L$ is non-trivial as soon as
$E \notin \langle \varepsilon_M^p \rangle$ and gives an odd component
of the whole $p$-class group of $L$.

\begin{theorem}\label{quintic}
As $t$ grows from $1$ up to $\BB$, each first occurrence of a square-free integer 
$M \geq 2$ in the factorization $m(t) := p^4 t^2 - 4s =: M r^2$, 
the degree $p-1$ cyclic imaginary subfield of $\Q(\sqrt M, \zeta_p)$, distinct from 
$\Q(\zeta_p)$, has a class number divisible by $p$, except possibly when the unit 
$E_s(p^2 t ) := \frac{1}{2} [p^2 t +  r \sqrt M)]$ is a $p$-th power in 
$\langle \varepsilon_M^{} \rangle$. 
\end{theorem}

\subsubsection{Lists of \texorpdfstring{$5$}{Lg}-class groups of
cyclic imaginary quartic fields}

The following program for $p=5$ verifies the claim with the above parametrized family
testing if $E_s(p^2 t )$ is a $p$-power in $\langle \varepsilon_M^{} \rangle$. For $p=5$, 
the mirror field $F_{5,M}$ is defined by the polynomial 
$${\sf P=x^4+5*M*x^2+5*M^2}, $$ 
still giving a particular faster program than the forthcoming one, 
valuable for any $p \geq 3$:

\smallskip
\ft\begin{verbatim}
LISTS OF 5-CLASS GROUPS OF QUARTIC FIELDS
{p=5;B=100;s=-1;Lp=List;Lh=List;p2=p^2;p4=p^4;for(t=1,B,
mt=p4*t^2-4*s;ut=p2*t/2;vt=1/2;C=core(mt,1);M=C[1];r=C[2];
res=Mod(M,4);D=quaddisc(M);w=quadgen(D);Y=quadunit(D);
if(res!=1,Z=ut+r*vt*w);if(res==1,Z=ut-r*vt+2*r*vt*w);z=1;n=0;
while(Z!=z,z=z*Y;n=n+1);P=x^4+5*M*x^2+5*M^2;K=bnfinit(P,1);
CK=K.cyc;C5=List;d=#CK;for(i=1,d,Cl=CK[d-i+1];
val=valuation(Cl,p);if(val>0,listinsert(C5,p^val,1)));L=List([M,C5]);
listput(Lh,vector(2,c,L[c])));Vh=vecsort(vector(B,c,Lh[c]),1,8);
print(Vh);print("#Vh = ",#Vh);
for(k=1,#Vh,if(Vh[k][2]==List([]),listput(Lp,Vh[k])));Vp=vecsort(Lp,1,8);
print("exceptions:",Vp)}
s=-1
[M,C5]=
[[89,[5]],[509,[5,5]],[626,[25,5]],[629,[5,5]],[761,[5]],[2501,[5]],[3554,[25]],
[5626,[5,5]],[5629,[5]],[10001,[5]],[15626,[5,5]],[15629,[25]],[22501,[5]],
[30626,[5,5]],[30629,[5]],[40001,[5]],[50626,[25,5]],[50629,[5]],[62501,[25,25]],
[75626,[5]],[75629,[5]],[90001,[5,5]],[105626,[125,25]],[105629,[5,5]],
(...)
[5175629,[125]],[5405629,[5]],[5640629,[5]],[5880629,[5]],[6125629,[5]]
#Vh = 100
exceptions:List([])

s=1
[M,C5]=
[[39,[5]],[51,[5]],[69,[5]],[114,[5]],[326,[5]],[434,[25]],[574,[5,5]],[674,[5]],[791,[5]],
[1086,[5]],[1111,[5,5]],[1406,[5]],[1761,[5]],[1914,[5,5]],[3981,[5]],[4171,[5,5]],
[5621,[5]],[8789,[5,5]],[10421,[5]],[11289,[5,5]],[13611,[5]],[14189,[5]],[15621,[25]],
[18906,[5]],[20069,[5,5]],[20501,[5,5]],[22499,[25,25]],
(...)
[4730621,[25,5]],[5405621,[5]],[5640621,[25]],[5880621,[5,5,5]],[6125621,[5]]]
#Vh = 100
exceptions:List([])
\end{verbatim}\ns

Taking $\BB = 200$ with $s=-1$ leads to the exceptional case ${\sf [29,[\ ]]}$.
For $s=1$ one gets the exceptional case ${\sf [21,[\ ]]}$.

\subsubsection{General program giving the \texorpdfstring{$p$}{Lg}-class group
of degree \texorpdfstring{$p-1$}{Lg} imaginary fields}

The following general program computes the defining polynomial $P$ of the algebraic 
number field $F_{p,M} := \Q \big((\zeta_p-\zeta_p^{-1}) \sqrt{M} \big)$; it tests if the unit $E_s(p^2 t)$ is
the $p$th power in $\langle \varepsilon_M^{} \rangle$, giving the list of exceptions.
One has to choose ${\sf p, \BB, s}$:

\smallskip
\ft\begin{verbatim}
LISTS OF p-CLASS GROUPS OF DEGREE p-1 IMAGINARY FIELDS I
{p=5;B=500;s=-1;Lp=List;Lh=List;Zeta=exp(2*I*Pi/p);p2=p^2;p4=p^4;
for(t=1,B,mt=p4*t^2-4*s;ut=p2*t/2;vt=1/2;C=core(mt,1);M=C[1];r=C[2];
res=Mod(M,4);D=quaddisc(M);w=quadgen(D);Y=quadunit(D);
if(res!=1,Z=ut+r*vt*w);if(res==1,Z=ut-r*vt+2*r*vt*w);z=1;n=0;
while(Z!=z,z=z*Y;n=n+1);P=1;for(i=1,(p-1)/2,A=(Zeta^i+ Zeta^-i-2)*M;
P=(x^2-A)*P);P=round(P);k=bnfinit(P,1);Ck=k.cyc;Cp=List;d=#Ck;
for(i=1,d,Cl=Ck[d-i+1];val=valuation(Cl,p);if(val>0,listinsert(Cp,p^val,1)));
L=List([M,Cp]);listput(Lh,vector(2,c,L[c])));Vh=vecsort(vector(B,c,Lh[c]),1,8);
print(Vh);print("#Vh = ",#Vh);
for(k=1,#Vh,if(Vh[k][2]==List([]),listput(Lp,Vh[k])));Vp=vecsort(Lp,1,8);
print("exceptions:",Vp)}
s=-1
[M,Cp]=
[[29,[]],[89,[5]],[509,[5,5]],[626,[25,5]],[629,[5,5]],[761,[5]],[2501,[5]],[3554,[25]],
[5626,[5,5]],[5629,[5]],[10001,[5]],[15626,[5,5]],[15629,[25]],[19109,[5]],[22061,[5,5]],
[22501,[5]],[30626,[5,5]],[30629,[5]],[40001,[5]],[42341,[5]],[50626,[25,5]],
[50629,[5]],[62501,[25,25]],[70429,[25]],[75626,[5]],[75629,[5]],[82234,[5]],
[90001,[5,5]],[105626,[125,25]],[105629,[5,5]],[122501,[5]],[140626,[5]],[140629,[5,5]],
(...)
[147015629,[5]],[148230629,[5]],[149450629,[5]],[150675629,[5,5,5]],
[151905629,[5]],[153140629,[5]],[154380629,[5]],[155625629,[5,5]]]
#Vh = 500
exceptions:List([[29,[])]])
s=1
[M,Cp]=
[21,[]],[39,[5]],[51,[5]],[69,[5]],[114,[5]],[326,[5]],[434,[25]],[514,[5]],[574,[5,5]],
[581,[5,5]],[674,[5]],[791,[5]],[874,[5]],[1086,[5]],[1111,[5,5]],[1191,[5]],[1351,[25]],
[1406,[5]],[1641,[5]],[1761,[5]],[1851,[5]],[1914,[5,5]],[2399,[5]],[2599,[25]],
[3251,[25]],[3981,[5]],[4171,[5,5]],[5474,[5]],[5621,[5]],[5774,[5]],[8294,[25,5]],
[8789,[5,5]],[10421,[5]],[11289,[5,5]],[13611,[5]],[14189,[5]],[15621,[25]],
(...)
[141015621,[5,5]],[142205621,[5,5]],[143400621,[25,5]],[144600621,[25,5]],
[145805621,[25]],[149450621,[5]],[150675621,[5,5]],[151905621,[5]],
[153140621,[5]],[155625621,[625,5]]
#Vh = 500
exceptions:List([[21,List([])]])
\end{verbatim}\ns

\smallskip
In this interval, all the $5$-class groups obtained are non-trivial, except for 
$s=-1$ and $M=29$, then for $s=1$ and $M=21$.
From Remark \ref{MB}, we compute: 
$$\log(155625629)/\log(5^4 \cdot 25 \cdot 10^4)
\approx 0.99978777. $$ 

Theorem \ref{heuristic} gives possible exceptions
up to $M_\BB^{\frac{1}{5}} =155625629^{\frac{1}{5}} \approx 43.49268545$.

\smallskip
One observes the spectacular decrease of counterexamples and the unique exception 
with $s=-1$, obtained for $t=151$, $p^2 t=25 \cdot 151 = 3775$, 
$m_{-1}(3775) = 701^2 \times 29$; whence the PARI data:
$${\sf Y=Mod(1/2*x + 5/2, x^2 - 29)},\ \ \ 
{\sf Z=Mod(2646275/2*x+14250627/2,x^2-29)}$$ 
(for $\varepsilon_{29}^{}$ and $E_{-1}(3775)$, respectively). One obtains easily 
the relation $E_{-1}(3775) = \varepsilon_{29}^{10}$.
The case $s=1$, $M=21$ is analogous.

\smallskip
Consider the case $p=7$, $s \in \{-1, 1\}$; exceptionally, we give the complete lists:

\ft\begin{verbatim}
p=7  B=100  s=-1
[M,Cp]=
[[37,[7]],[2402,[7]],[2405,[7]],[4706,[7]],[9605,[7]],[10357,[7]],
[11621,[49,7]],[21610,[7,7]],[21613,[7,7]],[38417,[7]],[60026,[7,7]],
[60029,[7]],[86437,[7,7]],[98345,[7]],[117653,[7]],[146077,[7]],
[153665,[7,7]],[177578,[7,7]],[194482,[7,7]],[194485,[49,7]],
[240101,[7]],[290522,[49]],[345745,[49]],[357365,[7]],[405770,[7,7]],
[405773,[49,7]],[470597,[7,7]],[540226,[7]],[540229,[7,7]],
[614657,[7,7]],[693890,[7,7]],[693893,[7]],[760733,[7,7,7]],
[866762,[7,7]],[866765,[7,7]],[960401,[7,7]],[1058842,[7]],
[1058845,[7,7]],[1162085,[49,7]],[1270130,[49,7,7]],[1270133,[7]],
[1382977,[7,7]],[1500626,[49]],[1500629,[7]],[1623077,[7]],
[1750330,[7]],[1882385,[7]],[2019242,[49]],[2019245,[7,7]],
[2160901,[7]],[2307362,[343]],[2307365,[7,7]],[2614690,[7,7]],
[2614693,[7]],[2775557,[7]],[2941226,[7]],[2941229,[49]],
[3111697,[7]],[3286970,[7]],[3286973,[7,7]],[3467045,[7]],
[3651922,[7]],[3841601,[7]],[4036082,[7]],[4036085,[49]],
[4235365,[7]],[4439453,[49]],[4648337,[49,7]],[4862026,[7,7]],
[4862029,[7]],[5080517,[7,7]],[5303810,[7]],[5303813,[7]],
[5531905,[7]],[5764802,[7,7]],[5764805,[7]],[6002501,[7]],
[6245005,[7,7]],[6744413,[49,7]],[7263029,[7]],[7800853,[7]],
[8357885,[7]],[9529573,[49,7,7]],[10144229,[7]],[10778093,[7,7]],
[11431165,[49]],[12103445,[49,7]],[12794933,[7]],[13505629,[7]],
[14235533,[7]],[14984645,[7]],[15752965,[7]],[16540493,[7]],
[17347229,[7]],[18173173,[7,7,7]],[19882685,[7]],[20766253,[7]],
[21669029,[7,7]],[22591013,[7]],[23532205,[7,7]]]
#Vh = 100
exceptions:List([])

p=7  B=100  s=1
[M,Cp]=
[[6,[7]],[741,[7,7]],[817,[7,7]],[1067,[7,7]],[1517,[49]],[2302,[49]],
[2397,[49]],[3477,[7]],[3603,[49,7]],[5402,[2401,7]],[5645,[7,7]],
[8070,[49]],[8441,[7,7]],[10421,[7]],[10842,[7,7]],[12155,[7]],
[13702,[7]],[15006,[49]],[21605,[7,7]],[27165,[7]],[35003,[7]],
[38415,[7]],[42803,[7]],[43637,[7]],[45085,[49]],[55319,[7]],
[56090,[7,7]],[63269,[7]],[64923,[7]],[68295,[7]],[70013,[7]],
[79383,[7]],[86435,[7]],[101442,[7]],[106711,[7]],[117645,[49,7]],
[144210,[49]],[153663,[7,7]],[163418,[7]],[194477,[7]],[216690,[7,7]],
[228245,[7]],[240099,[49,7]],[252255,[7,7]],[264710,[7]],
[290517,[49,7]],[308395,[7]],[345743,[7]],[437582,[7,7]],
[448453,[49,7]],[470595,[7]],[511797,[7]],[540221,[7]],
[640533,[7,7]],[693885,[7]],[735306,[49,7]],[777923,[7]],
[821742,[7]],[866757,[7]],[928653,[49]],[1058837,[49,7]],
[1162083,[7]],[1197565,[343]],[1215506,[7,7]],[1500621,[7,7]],
[1882383,[49,7]],[1927469,[7]],[2019237,[7]],[2160899,[7]],
[2407669,[7]],[2458623,[49]],[2614685,[7,7]],[2941221,[7,7]],
[3111695,[7]],[3651917,[7]],[3841599,[7,7]],[4439445,[7]],
[4648335,[7]],[4862021,[49,49,7]],[5080515,[7]],[5303805,[7]],
[5531903,[7]],[6002499,[7]],[6244997,[7]],[6744405,[7,7]],
[7263021,[7,7]],[7800845,[7]],[8934117,[7]],[9529565,[7]],
[10144221,[7]],[11431157,[7]],[13505621,[7]],[14984637,[7]],
[16540485,[7]],[18173165,[7]],[19018317,[7]],[19882677,[7]],
[20766245,[7]],[22591005,[7]],[23532197,[7]]]
#Vh = 100
exceptions:List([])
\end{verbatim}\ns

Of course, ${\sf \BB=100}$ is insufficient to give smaller Kummer radicals, but it is only a 
question of execution time and memory due to the instruction ${\sf bnfinit(P,1)}$
for ${\sf P}$ of degree $p-1$. 
It is clear that the same program for the \fop algorithm,
without computation of the $p$-class group, gives unlimited lists of degree $p-1$ 
imaginary cyclic fields with non-trivial $p$-class group, as soon as $M > M_\BB^\pow$ 
(cf. Theorem \ref{heuristic}):

\ft\begin{verbatim}
LISTS OF p-CLASS GROUPS OF DEGREE p-1 IMAGINARY FIELDS II
{p=7;B=10^5;s=1;LM=List;p4=p^4;for(t=1,B,mt=p4*t^2-4*s;
M=core(mt);L=List([M]);listput(LM,vector(1,c,L[c])));
VM=vecsort(vector(B-(1+s),c,LM[c]),1,8);
print(s);print(VM)}
s=-1
[M]=
[37],[53],[74],[149],[554],[1373],[2237],[2402],[2405],[3026],[3242],[4706],[5882],
[7373],[9605],[10357],[11621],[18229],
(...)
[24006638717653],[24007599060029],[24008559421613],[24009519802405]]
s=1
[M]=
[5],[6],[101],[145],[206],[215],[570],[629],[663],[731],[741],[817],[887],[894],[1067],
[1207],[1389],[1517],[1893],[2085],[2162],
(...)
[24004718090517],[24005678394477],[24006638717645],[24008559421605]
\end{verbatim}\ns

\end{document}